\documentclass[11pt]{article}
\usepackage[left=1in,right=1in,top=1in,bottom=1in]{geometry}

\usepackage{graphicx}
\usepackage{amsfonts}
\usepackage{url}
\usepackage{amsmath,amssymb}
\usepackage{enumerate}
\usepackage{color}
\usepackage{amsthm,amsmath}
\usepackage{bbm}
\usepackage{xspace}
\usepackage{multicol}
\usepackage{subcaption}
\usepackage{rotating}
\usepackage{appendix}
\usepackage{tikz-cd}
\tikzcdset{every label/.append style = {font = \small}}
\usepackage{tikz}
\usepackage{mathrsfs}  

\usepackage{fancybox}
\usepackage{tikz}
\usetikzlibrary{arrows,shapes,positioning}
\usepackage{tikz-cd}

\usepackage[colorlinks]{hyperref}

\begin{document}

\newcommand{\Sp}{\mathcal{S}}
\newcommand{\C}{\mathcal{C}}
\newcommand{\R}{\mathcal{R}}

\newtheorem{thm}{Theorem}
\newtheorem{pro}[thm]{Proposition}
\newtheorem{lem}[thm]{Lemma}
\newtheorem{cor}[thm]{Corollary}

\theoremstyle{definition}
\newtheorem{dfn}[thm]{Definition}
\newtheorem{exa}[thm]{Example}
\newtheorem{rem}[thm]{Remark}

\newcommand{\tr}{\intercal}

\title{A two-strain model of infectious disease spread with asymmetric temporary immunity periods and partial cross-immunity}

\author{%
  Matthew D. Johnston$^{*}$, 
  Bruce Pell, 
  and 
  David A. Rubel \\ \\ Department of Mathematics + Computer Science\\Lawrence Technological University, 21000 W 10 Mile Rd, Southfield, MI 48075, USA\\ *Corresponding author: \tt{mjohnsto1@ltu.edu}
}

\date{}

\maketitle

\begin{abstract}
We introduce a two-strain model with asymmetric temporary immunity periods and partial cross-immunity. We derive explicit conditions for competitive exclusion and coexistence of the strains depending on the strain-specific basic reproduction numbers, temporary immunity periods, and degree of cross-immunity. The results of our bifurcation analysis suggest that, even when two strains share similar basic reproduction numbers and other epidemiological parameters, a disparity in temporary immunity periods and partial or complete cross-immunity can provide a significant competitive advantage. To analyze the dynamics, we introduce a quasi-steady state reduced model which assumes the original strain remains at its endemic steady state. We completely analyze the resulting reduced planar hybrid switching system using linear stability analysis, planar phase-plane analysis, and the Bendixson-Dulac criterion. We validate both the full and reduced models with COVID-19 incidence data, focusing on the Delta (B.1.617.2), Omicron (B.1.1.529), and Kraken (XBB.1.5) variants. These numerical studies suggest that, while early novel strains of COVID-19 had a tendency toward dramatic takeovers and extinction of ancestral strains, more recent strains have the capacity for co-existence. 
\end{abstract}

\section{Introduction}


Since the start of the COVID-19 pandemic, compartmental differential equation-based mathematical models have played an instrumental role in informing public health decisions. Early on, modeling efforts focused on forecasting and evaluating the efficacy of non-pharmaceutical interventions such as lockdowns \cite{baba2021,sardar2020}, school closures \cite{wang2020,wu2022}, contact tracing \cite{biala2022,bilinski2020}, and face mask utilization \cite{eikenberry2020,ngonghala2020}. Following the rapid development of effective vaccines and antiviral treatments, the emphasis shifted to modeling pharmaceutical measures, such as estimating the effects of different vaccine distribution strategies \cite{foy2021,liu2022}, evaluating vaccine efficacy \cite{johnston2022,lin2022}, and assessing waning vaccine immunity \cite{pell2023}. Despite the unprecedented global effort to eradicate COVID-19, however, the disease has continued to spread widely and is now widely considered to be endemic in the global population.

A significant factor in this continuing spread has been the emergence of variants, such as Delta (B.1.617.2), Omicron (B.1.1.529), and the Omicron subvariant Kraken (XBB.1.5), which were reported as the most transmissible strains at the time. Transmissibility, however, is only one of several factors, including vaccine-resistance, diagnostic evasion, and cross-immunity, which cumulatively determine whether emerging strains exhibit \emph{competitive exclusion} or \emph{coexistence} with the circulating ancestral strains. Competitive exclusion, where one strain drives the others to extinction, has been observed with the Delta and Omicron waves during the COVID-19 pandemic, where earlier strains were largely eradicated over the course of several months. Indeed, such dramatic takeovers are expected during the beginning of a pandemic by a novel virus as it is rapidly mutating and exploring its evolutionary fitness. As COVID-19 transitions from a pandemic to endemic infection, however, it is anticipated that we will see multiple co-existing strains circulating in a given year, akin to seasonal influenza \cite{fudolig2020}.


Mathematically, it has been common to account for multiple strains in an SIR-type (Susceptible-Infectious-Removed) model of infectious disease spread \cite{Kermack1927} by dividing the infectious class between several competing and non-overlapping strains. The roots of this research can be found in several areas, including the study of viruses such as influenza, bacterial infections, and parasites \cite{andreasen1997,hethcote2000,lin1999,white1998}. For many basic SIR-type models extended to multiple strains, the only possibility is competitive exclusion \cite{andreasen1997,ciupeanu2022,wang2022,wang2022-2}. However, there are several mechanisms known to facilitate coexistence, including co-infection, cross-immunity, seasonal variations, and age-stratification of infectious \cite{martcheva2015}. In the context of COVID-19, several studies have demonstrated the capacity of coexistence in certain cases, such as utilizing different force of infection terms \cite{khyar2020,wang2022-3}, and strain-specific vaccination \cite{fudolig2020}, and cross-immunity \cite{pell2023}. Understanding the capacity of a multistrain model to exhibit competitive exclusion or coexistence remains challenging, however, due to the high-dimensionality of the associated models.






In this paper, we introduce a two-strain model which incorporates asymmetric temporary immunity periods and partial cross-immunity. We determine explicit conditions for competitive exclusion and coexistence of the two strains based on the basic reproduction numbers, the temporary immunity period durations, and the degree of cross-immunity. To address the challenges of analyzing the dynamics of multistrain models, we further introduce a dimension-reducing modeling framework for analyzing the emergence of new strains of a virus. In our method, we assume that the original strain is endemic in the population and remains at steady state throughout the evolution of the new strain. This allows the dynamics of the emerging strain to be reduced while closely capturing the dynamics and long-term behavior of the full model. The method is based on the quasi-steady state approximation (QSSA) which is extensively used in the study of biochemical reaction networks to justify the Michaelis-Menten \cite{ingalls2013,michaelis1913} and Hill \cite{hill1910} rate forms. More broadly in mathematical biology, the use of the QSSA is more limited but includes application to a two-strain tuberculosis model \cite{castillo1997}, parasite-host models \cite{wang2011}, cancer therapy \cite{rutter2017}, vector-borne illnesses \cite{rashkov2019}, and a two-strain dengue fever model \cite{rashkov2021}. We validate the models by parameter-fitting to COVID-19 incidence data in the United States across multiple strains, including the Delta, Omicron, and Kraken variants. The results suggest that differences in the temporary immunity periods and partial cross-immunity are sufficient to account for the dramatic shifts in variant proportions we have seen during the course of the COVID-19 pandemic.




\section{Main Results}


\subsection{Basic Two-Strain Model}
\label{sec:basic}

Consider the following two-strain compartmental model constructed after the classical SIR (Susceptible-Infectious-Removed) model introduced by Kermack and McKendrick in 1927 \cite{Kermack1927}:
\begin{equation}
    \small
    \label{SIR}
    \begin{tikzcd}
    \mbox{\fbox{\begin{tabular}{c} Infectious \\ \emph{Original Strain} \\($I_1$)\end{tabular}}} \arrow[rr,"\gamma_1"] & & \mbox{\fbox{\begin{tabular}{c} Immune \\ \emph{Original Strain} \\($R_1$)\end{tabular}}} \arrow[ld,"\sigma_1"']  \\[-0.4in]
    & \mbox{\fbox{\begin{tabular}{c} Susceptible \\($S$)\end{tabular}}}  \arrow[lu,"\beta_1"'] \arrow[dr,"\beta_2"] & \\[-0.4in]
    \mbox{\fbox{\begin{tabular}{c} Immune \\ \emph{Emerging Strain} \\($R_2$)\end{tabular}}} \arrow[ru,"\sigma_2"] & & \mbox{\fbox{\begin{tabular}{c} Infectious \\ \emph{Emerging Strain} \\($I_2$)\end{tabular}}} \arrow[ll,"\gamma_2"'] \\[-0.2in]
\end{tikzcd}
\end{equation}
In the model \eqref{SIR}, we allow each strain to infect a common pool of susceptible individuals at different rates ($\beta_1$ and $\beta_2$, respectively, for the original and emerging strain). We also allow each strain to exhibit different mean infection periods ($\gamma_1^{-1}$ and $\gamma_2^{-1}$) and different mean temporary immunity periods ($\sigma_1^{-1}$ and $\sigma_2^{-1}$). We do not allow co-infection and assume complete cross-immunity so that individuals catch one strain at a time and completely recover before they are capable of being infected by either strain again. The model \eqref{SIR} is a special case of those considered in \cite{wang2022,wang2022-2}.

The model \eqref{SIR} can be corresponded to the following system of ordinary differential equations:
\begin{equation}
\label{SIR-DE}
\left\{ \; \; \;
    \begin{aligned}
    \frac{dS}{dt} & = - \frac{S}{N} \left( \beta_1 I_1 + \beta_2 I_2 \right) + 
    \sigma_1 R_1 + \sigma_2 R_2, & & \\
    \frac{dI_1}{dt} & = \frac{\beta_1}{N} S I_1 - \gamma_1 I_1, & \frac{dI_2}{dt} & = \frac{\beta_2}{N} S I_2 - \gamma_2 I_2, \\
    \frac{dR_1}{dt} & = \gamma_1 I_1 - \sigma_1 R_1, &     \frac{dR_2}{dt} & = \gamma_2 I_2 - \sigma_2 R_2.
    \end{aligned}
\right.
\end{equation}
Note that the equations \eqref{SIR-DE} imply that the total population is constant, i.e. $N = S + I_1 + R_1 + I_2 + R_2$. Setting $S = N - I_1 - R_1 - I_2 - R_2$, we reduce the system to:
\begin{equation}
    \label{SIR-DE-reduced}
    \left\{ \; \; \;
    \begin{aligned}
    \frac{dI_1}{dt} & = \frac{\beta_1}{N} (N - I_1 - R_1 - I_2 - R_2) I_1 - \gamma_1 I_1, & \frac{dI_2}{dt} & = \frac{\beta_2}{N} (N - I_1 - R_1 - I_2 - R_2) I_2 - \gamma_2 I_2, \\
    \frac{dR_1}{dt} & = \gamma_1 I_1 - \sigma_1 R_1, &     \frac{dR_2}{dt} & = \gamma_2 I_2 - \sigma_2 R_2.
    \end{aligned}
    \right.
\end{equation}
It can be easily computed that the system \eqref{SIR-DE-reduced} only permits the following steady states, where $\mathbf{x} = (S, I_1, R_1, I_2, R_2)$:
\begin{eqnarray}
\label{dfe}
\mathbf{x}_0 & = & (N, 0, 0, 0, 0) \\
\label{ose}
\mathbf{x}_1 & = & \displaystyle{\left(\frac{N \gamma_1}{\beta_1}, \frac{N \sigma_1 (\beta_1 - \gamma_1)}{\beta_1(\gamma_1 + \sigma_1)},\frac{N \gamma_1(\beta_1 - \gamma_1)}{\beta_1(\gamma_1 + \sigma_1)},0,0\right)} \\
\label{ese}
\mathbf{x}_2 & = & \displaystyle{\left(\frac{N \gamma_2}{\beta_2}, 0,0,\frac{N \sigma_2 ( \beta_2 - \gamma_2)}{\beta_2( \gamma_2 + \sigma_2)},\frac{N \gamma_2(\beta_2 - \gamma_2)}{\beta_2( \gamma_2 + \sigma_2)}\right)}
\end{eqnarray}
The \emph{disease free steady state} $\mathbf{x}_0$ exists for all parameter values, while the \emph{original strain only steady state} $\mathbf{x}_1$ is physically relevant if and only if $\beta_1 > \gamma_1$ and the \emph{emerging strain only steady state} $\mathbf{x}_2$ is physically relevant if and only if $\beta_2 > \gamma_2$. The system does not have the capacity for a \emph{coexistence steady state}, i.e. a steady state $\mathbf{x}_{12}$ with $I_1 > 0$ and $I_2 > 0$.

\subsection{Basic Reproduction Number}

The basic reproduction number of a disease, denoted $\mathscr{R}_0$, is one of the most important and well-studied parameters in the study of infectious disease spread. Intuitively, it corresponds to the expected number of secondary infections produced by a single primary infection in a fully susceptible population \cite{diekmann1990}. Consequently, an infectious disease has the capacity to infiltrate a population if and only if $\mathscr{R}_0 > 1$. For multi-strain models like \eqref{SIR}, however, we are also interested in the capacity of an individual \emph{strain} to infiltrate a population, either in the absence or presence of other strains. Following the notation of \cite{VANDENDRIESSCHE2002}, we let the basic reproduction number of strain $i$ in the absence of other strains be denoted by $\mathscr{R}_i$, and the basic reproduction number of strain $i$ in the presence of strain $j$ be denoted by $\mathscr{R}_{ij}$. The threshold $\mathscr{R}_i > 1$ suggests that strain $i$ has the capacity to infiltrate the population in the absence of other strains, while the condition $\mathscr{R}_{ij} > 1$ suggests that strain $i$ has the capacity to infiltrate a population already infected at a steady level with strain $j$. 


For compartmental differential equations models, the next-generation method can be used to calculate the basic reproduction numbers \cite{diekmann2009,heffernan2005,VANDENDRIESSCHE2002,van2008,van2017} (see Appendix \ref{app:ngm} for details). For the basic two-strain model \eqref{SIR-DE-reduced}, the next generation method gives the parameters (see Appendix \ref{app:basic} for details):
\[
\mathscr{R}_1 = \displaystyle{\frac{\beta_1}{\gamma_1}}, \mathscr{R}_2 = \displaystyle{\frac{\beta_2}{\gamma_2}}, \mathscr{R}_{12} = \displaystyle{\frac{\mathscr{R}_1}{\mathscr{R}_2}}, \mathscr{R}_{21} = \frac{\mathscr{R}_2}{\mathscr{R}_1}, \mbox{ and }\displaystyle{\mathscr{R}_0 = \max\{ \mathscr{R}_1, \mathscr{R}_2 \}}.    
\]
It follows from $\mathscr{R}_i > 1$, $i=1,2,$ that strain $i$ has the capacity to infiltrate a disease-free population only if $\beta_i > \gamma_i$, and from $\mathscr{R}_0 = \max\{ \mathscr{R}_1, \mathscr{R}_2 \} > 1$ that the disease will infiltrate the population if at least one strain is able to infiltrate. The condition $\mathscr{R}_{ij} > 1$ suggests the strain $i$ has the capacity to infiltrate a population already infected with strain $j$ if $\mathscr{R}_i > \mathscr{R}_j$, i.e. its basic reproduction number must be strictly higher than that of the other strain. 

It is furthermore known from the results of \cite{wang2022,wang2022-2} that: (a) the disease free steady state $\mathbf{x}_0$ \eqref{dfe} is stable if $\mathscr{R}_1 < 1$ and $\mathscr{R}_2 < 1$; (b) the original strain only steady state $\mathbf{x}_1$ \eqref{ose} is stable if $\mathscr{R}_1 > 1$ and $\mathscr{R}_1 > \mathscr{R}_2$; and (c) the emerging strain only steady state $\mathbf{x}_2$ \eqref{ese} is stable if $\mathscr{R}_2 > 1$ and $\mathscr{R}_2 > \mathscr{R}_1$. Which individual strains survives is only dependent on the relative values of the basic reproduction numbers $\mathscr{R}_1$ and $\mathscr{R}_2$. Ultimately, the more contagious strain will infiltrate the population and eliminate the other strain; coexistence of strains is not permitted in the long-term dynamics.




\subsection{Full Asymmetric Temporary Immunity Periods and Partial Cross-Immunity Model}

We now extend the two-strain model \eqref{SIR} to include asymmetric temporary immunity periods and partial cross-immunity:
\begin{equation}
    \small
    \label{SIR2}
    \begin{tikzcd}
    \mbox{\fbox{\begin{tabular}{c} Infectious \\ \emph{Original Strain} \\($I_1$)\end{tabular}}} \arrow[rr,"\gamma_1"] & & \mbox{\fbox{\begin{tabular}{c} Immune \\ \emph{Original Strain} \\($R_1$)\end{tabular}}} \arrow[ld,"\sigma_1"'] \arrow[dd,"\beta_2"] \\[-0.4in]
    & \mbox{\fbox{\begin{tabular}{c} Susceptible \\($S$)\end{tabular}}}  \arrow[lu,"\beta_1"'] \arrow[dr,"\beta_2"] & \\[-0.4in]
    \mbox{\fbox{\begin{tabular}{c} Immune \\ \emph{Emerging Strain} \\($R_2$)\end{tabular}}} \arrow[uu,"\beta_1(1-\epsilon)"'] \arrow[ru,"\sigma_2"] & & \mbox{\fbox{\begin{tabular}{c} Infectious \\ \emph{Emerging Strain} \\($I_2$)\end{tabular}}} \arrow[ll,"\gamma_2"'] \\[-0.2in]
\end{tikzcd}
\end{equation}
In the model \eqref{SIR2}, we assume individuals in the temporary immunity states $R_1$ and $R_2$ can only be infected by the \emph{other} strain. Those recovered from the original strain ($R_1$) are infected by the emerging strain at the same rate as susceptible individuals ($S$); however, those recovered from the emerging strain ($R_2$) are provided with partial protection from catching the original strain. This asymmetric cross-immunity can occur when the antibodies produced from infection by the emerging strain are sufficient protection against the original strain, but not the other way around. The degree of cross-immunity is tuned by the parameter $0 \leq \epsilon \leq 1$, with $\epsilon = 0$ corresponding to no cross-immunity, and $\epsilon = 1$ corresponding to complete cross-immunity.

Note that we have opted not to include in \eqref{SIR2} a latency period, which is often incorporated in the form of an exposed class of individuals. The decision to not include a latency period was made for mathematical tractability and to allow us to focus on other key aspects of the disease dynamics, such as transmission rates, immunity periods, and the impact of emerging variants. While the latency period is an important component of the disease process, its omission does not significantly alter the general dynamics of the model leading to similar results for the reproduction numbers \cite{van2017,wang2022-2}. 
However, we note that the latency period plays a role in the understanding of disease transmission and can influence the effectiveness of control measures, such as testing and lockdown protocols. 

We utilize the following system of ODEs to model the time evolution of the full model \eqref{SIR2}:
\begin{equation}
    \label{SIR-partial}
    \left\{ \; \; \;
    \begin{aligned}
    \frac{dS}{dt} & = - \frac{S}{N} \left( \beta_1 I_1 + \beta_2 I_2 \right) +  \sigma_1 R_1 + \sigma_2 R_2, & & \\
    \frac{dI_1}{dt} & = \frac{\beta_1}{N} I_1 (S + (1-\epsilon) R_2) - \gamma_1 I_1, & \frac{dI_2}{dt} & = \frac{\beta_2}{N} I_2 (S + R_1) - \gamma_2 I_2, \\
    \frac{dR_1}{dt} & = \gamma_1 I_1 - \sigma_1 R_1 - \frac{\beta_2}{N} R_1 I_2, & \frac{dR_2}{dt} & = \gamma_2 I_2 - \sigma_2 R_2 - \frac{\beta_1}{N}  (1 - \epsilon) R_2 I_1.\\
    \end{aligned}
    \right.
\end{equation}
Substituting $S = N - I_1 - R_1 - I_2 - R_2$ into \eqref{SIR-partial}, we obtain the following equivalent system of ODEs:
\begin{equation}
    \label{SIR-partial-reduced}
    \left\{ \; \; \;
    \begin{aligned}
    \frac{dI_1}{dt} & = \frac{\beta_1}{N} (N - I_1 - R_1 - I_2 - \epsilon R_2) I_1 - \gamma_1 I_1, & \frac{dI_2}{dt} & = \frac{\beta_2}{N} (N - I_1 - I_2 - R_2) I_2 - \gamma_2 I_2, \\
    \frac{dR_1}{dt} & = \gamma_1 I_1 - \sigma_1 R_1 - \frac{\beta_2}{N} R_1 I_2, & \frac{dR_2}{dt} & = \gamma_2 I_2 - \sigma_2 R_2 - \frac{\beta_1}{N} (1 - \epsilon) R_2 I_1.
    \end{aligned}
    \right.
\end{equation}

The analysis of \eqref{SIR-partial-reduced} is more complicated than that of \eqref{SIR-DE-reduced}. Nevertheless, it can be easily checked that the three steady states of \eqref{SIR-DE-reduced} from Section \ref{sec:basic}, \eqref{dfe}-\eqref{ese}, are also steady states of \eqref{SIR-partial-reduced}. The basic reproductive numbers listed below \eqref{r} can be determined using the next generation method \cite{VANDENDRIESSCHE2002} (see Appendix \ref{app:new} for details):
\begin{equation}\small
\label{r}
\mathscr{R}_1 = \displaystyle{\frac{\beta_1}{\gamma_1}}, \mathscr{R}_2 = \displaystyle{\frac{\beta_2}{\gamma_2}}, \mathscr{R}_{12} = \frac{\mathscr{R}_1}{\mathscr{R}_2} \left( \frac{(1-\epsilon)\beta_2 + \epsilon \gamma_2  + \sigma_2}{\gamma_2 + \sigma_2 } \right), \mathscr{R}_{21} =  \frac{\mathscr{R}_2}{\mathscr{R}_1} \left( \frac{\beta_1 + \sigma_1}{\gamma_1 + \sigma_1} \right),  \mbox{ and } \mathscr{R}_0 = \max \{ \mathscr{R}_1, \mathscr{R}_2 \}.
\end{equation}

The condition $\mathscr{R}_{21} > 1$, required for the emerging strain to infiltrate a population already infected at an endemic level with the original strain, can be intuitively interpreted by considering limiting cases. If the original strain has a short temporary immunity period (i.e. $\sigma_1 \to \infty$, $\sigma_1^{-1} \to 0$), the emerging strain needs to be more contagious than the original strain to gain a foothold in the population ($\mathscr{R}_2 > \mathscr{R}_1$). This is due to the emerging strain constantly having to compete with the original strain for new infections. If the original strain has a longer temporary immunity period (i.e. $\sigma_1 \to 0$, $\sigma_1^{-1} \to \infty$), however, the emerging strain only needs to be able to sustain itself in the population on its own to survive ($\mathscr{R}_2 > 1$). This is due to the emerging strain having exclusive ability to infect those who have previously been infected with the original strain.

Similar intuition holds for the condition $\mathscr{R}_{12} > 1$ required for the original strain to survive as the emerging strain infiltrates the population. If the immunity period of the emerging strain is short (i.e. $\sigma_2 \to \infty$, $\sigma_2^{-1} \to 0$), or the degree of cross-immunity provided by the emerging strain is high (i.e. $\epsilon \to 1$), then the original strain will survive only if it is more contagious than the emerging strain ($\mathscr{R}_1 > \mathscr{R}_2$). If, however, the immunity period of the emerging strain is long (i.e. $\sigma_2 \to 0$, $\sigma_2^{-1} \to \infty$) then the original strain will survive only if it can survive on its own ($\mathscr{R}_1>1$).

Unlike the basic model \eqref{SIR}, the model \eqref{SIR2} also allows for a co-existence steady state where both strains survive and circulate in the population, so long as specific conditions on the basic reproduction numbers, temporary immunity periods, and degree of cross-immunity are satisfied. This is more formally stated in the following result, which we prove in Appendix \ref{app:a}.

\begin{thm}
\label{thm-coexistence}
The full two-strain model with asymmetric temporary immunity periods and partial cross-immunity \eqref{SIR-partial-reduced} has a coexistence steady state (i.e. $\mathbf{x}_{12} = (S,I_1,R_1,I_2,R_2)$ with $I_1 >0$ and $I_2 > 0$) if and only if $\min\{ \mathscr{R}_1, \mathscr{R}_2, \mathscr{R}_{12}, \mathscr{R}_{21} \} > 1$, where $\mathscr{R}_1, \mathscr{R}_2, \mathscr{R}_{12},$ and $\mathscr{R}_{21}$ are the basic reproduction numbers from \eqref{r}. Furthermore, this coexistence steady state is unique whenever it exists.
\end{thm}


\begin{table}[t!]
    \centering
    \begin{tabular}{l|c|l}
    \hline \hline
    Variable & Units & Description \\
    \hline \hline
    $S \geq 0$ & people & Susceptible individuals \\
    $I_i \geq 0$ & people & Infectious individuals ($i^{th}$ strain) \\
    $R_i \geq 0$ & people & Temporarily immune individuals ($i^{th}$ strain)\\
    $t \geq 0$ & days & Time elapsed \\
    \hline \hline
    Parameter & Units & Description  \\
    \hline \hline
    $\beta_i \geq 0$ & days$^{-1}$ & Transmission rate ($i^{th}$ strain) \\
    $\gamma_i^{-1} \geq 0$ & days & Infectious period ($i^{th}$ strain)\\
    $\sigma_i^{-1} \geq 0$ & days & Temporary immunity period ($i^{th}$ strain) \\
    $0 \leq \epsilon \leq 1$ & --- & Degree of cross-immunity \\
    $\mathscr{R}_0 \geq 0$ & --- & Basic reproduction number of disease \\
    $\mathscr{R}_i \geq 0$ & --- & Basic reproduction number of strain $i$ \\
    $\mathscr{R}_{ij} \geq 0$ & --- & Basic reproduction number of strain $i$ in presence of strain $j$ \\
    \hline \hline
    \end{tabular}
    \caption{\small Variables and parameters for the full model \eqref{SIR-partial-reduced} and reduced model \eqref{SIR-piecewise}.}
    \label{table1}
\end{table}



\subsection{Reduced Asymmetric Temporary Immunity Periods and Partial Cross-Immunity Model}

Although we are able to determine conditions for the existence of four steady states of \eqref{SIR-partial-reduced}, analyzing the dynamics remains difficult to perform directly due to the nonlinearities, four-dimensional state space, and seven undetermined parameters. We are also not able to write down an explicit closed form for the co-existence steady state. This makes linear stability and bifurcation analysis challenging.

To make analysis of the model's dynamics more tractable, we perform a model reduction of \eqref{SIR-partial-reduced}. We are primarily interested in the situation where the original strain has become endemic in the population before the emerging strain arrives. This suggests the modeling assumption that the original strain remains at its endemic steady state throughout the dynamics of the emerging strain. To update the dynamics of the emerging strain, we substitute the steady state values of $I_1$ and $R_1$ into the dynamical equations for $I_2$ and $R_2$. This produces a system of two differential equations for the emerging strain $I_2$ and $R_2$, and two algebraic equations for the original strain $I_1$ and $R_1$. Consequently, we assume that the first two equations in \eqref{SIR-partial-reduced} are at steady state. This gives the following system of equations:
\begin{equation}
    \label{eq10}
    \left\{ \; \; \;
    \begin{aligned}
    & \frac{\beta_1}{N} (N - I_1 - R_1 - I_2 - \epsilon R_2) I_1 - \gamma_1 I_1 = 0, \\
    & \gamma_1 I_1 - \sigma_1 R_1 - \frac{\beta_2}{N} R_1 I_2 = 0.
    \end{aligned}
    \right.
\end{equation}
Solving \eqref{eq10} gives rise to the following function, which tracks the steady state level of $I_1$ as a function of the infection level of $I_2$ and $R_2$:
\begin{equation}
    \label{omega}
    \omega(I_2,R_2) = \left\{
    \begin{aligned}
        & \frac{(N(\beta_1 - \gamma_1) - \beta_1 I_2 - \beta_1 \epsilon R_2)(\beta_2 I_2 + N\sigma_1)}{\beta_1 (\beta_2 I_2 + N(\gamma_1 + \sigma_1))}, & & \mbox{   if } I_2 + \epsilon R_2 <  N \left(1 - \frac{1}{\mathscr{R}_1}\right) \\
        & 0, & & \mbox{   if } I_2 + \epsilon R_2 \geq N \left(1 - \frac{1}{\mathscr{R}_1}\right).
    \end{aligned}
    \right.
\end{equation}
We substitute the steady state function $I_1 = \omega(I_2,R_2)$ into \eqref{SIR-partial-reduced} to get the reduced two-strain with asymmetric temporary immunity periods and partial cross-immunity model: 
\begin{equation}
    \label{SIR-piecewise}
    \left\{
    \begin{aligned}
    \frac{dI_2}{dt} & = \frac{\beta_2}{N} (N - \omega(I_2,R_2) - I_2 - R_2) I_2 - \gamma_2 I_2, \\
    \frac{dR_2}{dt} & = \gamma_2 I_2 - \sigma_2 R_2 - \frac{\beta_1}{N} (1-\epsilon) \omega(I_2,R_2) R_2.
    \end{aligned}
    \right.
\end{equation}
Notice that \eqref{SIR-piecewise} is a planar system which only depends on the emerging strain ($I_2$ and $R_2$). The original strain is tracked through the steady state function $I_1 = \omega(I_2,R_2)$ \eqref{omega}.

Since $\omega(I_2,R_2)$ is a piecewise-defined function, the system \eqref{SIR-piecewise} is a state-dependent switching system \cite{LIU2013}. The dynamics of such systems are more varied than standard dynamical systems but have been increasingly studied in recent years due to their applications in control theory \cite{egerstedt2003,sun2011,zhu2015}. We note, in particular, that the right-hand side of \eqref{SIR-piecewise} is continuous everywhere but not differentiable at $\displaystyle{I_2 + \epsilon R_2 = N \left(1 - \frac{1}{\mathscr{R}_1}\right)}$. The system \eqref{SIR-piecewise} is planar and consequently its dynamics can be analyzed by methods such as linear stability analysis, phase-plane analysis, and the Bendixson-Dulac criterion \cite{burton2005}. 

The method used to reduce \eqref{SIR-partial-reduced} to \eqref{SIR-piecewise} mirrors that of the quasi-steady state assumption (QSSA). The QSSA is used when there is a parametrizable time-scale separation between two parts of a process and has been popularly used in biochemistry to justify the Michaelis-Menten \cite{michaelis1913} and Hill \cite{hill1910} kinetic rate functions. This and other asymptotic methods also has been used to reduce models in mathematical epidemiology and other areas of mathematical biology in order to determine the long-term behavior of complicated models \cite{castillo1997,feliu2013,feliu2022,rashkov2021,rashkov2019,rutter2017,thieme1992,wang2011}. For \eqref{SIR-partial-reduced} and \eqref{SIR-piecewise}, however, we do not assume that there is a parametrizable time-scale separation between the original and emerging strain; in fact, such an assumption would be poorly justified for competing strains of COVID-19 as the infectivity and temporary immunity periods are on similar orders of magnitude. Consequently, we cannot treat the reduced system \eqref{SIR-piecewise} as a asymptotically limiting case of \eqref{SIR-partial-reduced}. Nevertheless, numerical simulations show that the full system \eqref{SIR-partial-reduced} and reduced system \eqref{SIR-piecewise} have comparable dynamics when the full system \eqref{SIR-partial-reduced} is assumed to start near the endemic steady state (see Figure \ref{fig:phaseplane}). We leave further consideration of the relationship between the behaviors of the full model \eqref{SIR-partial-reduced} and reduced model \eqref{SIR-piecewise} as future work.

\subsection{Reduced Model Analysis (All Parameters)}

We now turn our attention to the dynamical behavior of the switching system of differential equations \eqref{SIR-piecewise}. Firstly, we want to ensure that the model is well-behaved and that solutions remain physically meaningful at all times. This is verified by the following result.

\begin{thm}
\label{theorem1}
Solutions to the reduced two-strain model \eqref{SIR-piecewise} starting in the following compact set $\Lambda$ exist, are unique, and remain in $\Lambda$ for all time $t \geq 0$:
\begin{equation}
    \label{trapping}
    \Lambda = \{ (I_2,R_2) \in \mathbb{R}^2_{\geq 0} \; | \; I_2 + R_2 \leq N \}.
\end{equation}
\end{thm}

\begin{proof}

We first show that $\Lambda$ \eqref{trapping} is a trapping region of \eqref{SIR-piecewise} with three boundaries: $I_2=0$, $R_2=0$, and $I_1 + R_2 = N$. Consider the boundary $I_2=0$. Substituting $I_2=0$ into \eqref{SIR-piecewise} gives $I_2'=0$. It follows that the $R_2$-axis is an invariant set. Therefore, no solution may pass through $I_2=0$. 
Next, consider the boundary $R_2=0$. Substituting $R_2=0$ into \eqref{SIR-piecewise} gives $R_2'= \gamma_2 I_2 > 0$ for $I_2  > 0$. It follows that trajectories cannot excape the positive quadrant through $R_2=0$. 
Lastly, consider the boundary $I_2+R_2=N$ and the function $L(t) = I_2(t) + R_2(t)$. Along trajectories $(I_2(t),R_2(t))$ of \eqref{SIR-partial-reduced} we have that
\[\begin{aligned} L'(t) = I_2'(t)+R_2'(t) & = \frac{\beta_2}{N} (N - \omega(I_2,R_2) - I_2 - R_2) I_2 - \sigma_2 R_2 - \frac{\beta_1}{N} (1 - \epsilon) \omega(I_2,R_2) R_2\\
& = - \frac{\beta_2}{N} \omega(I_2,R_2) I_2 - \sigma_2 R_2 - \frac{\beta_1}{N} (1 - \epsilon) \omega(I_2,R_2) R_2
\end{aligned}\]
where we have used the consideration that we are only interested in $I_2+R_2 = N$ to reduce the first term. Since $\omega(I_2,R_2) \geq 0$ by \eqref{omega}, we have that $L'(t) < 0$ whenever $L(t) = N$. Consequently, trajectories starting in the set $\Lambda$ given by \eqref{trapping} remain in $\Lambda$, and we are done.

To prove existence and uniqueness with $\Lambda$, we note that the vector field \eqref{SIR-piecewise} is continuous in $\Lambda$ and that $\Lambda$ is closed and bounded and therefore a compact set. It follows that the vector field is Lipschitz continuous within $\Lambda$ which implies that solutions exist and are unique, and we are done.
\end{proof}

\subsection{Reduced Model Analysis ($\epsilon=0$ or $\epsilon=1$)}

We now consider the dynamics of \eqref{SIR-piecewise} within the invariant set $\Lambda$ in the specific cases of $\epsilon = 0$ (no cross-immunity) and $\epsilon = 1$ (full cross-immunity). Note that the system is planar so that it is sufficient to consider properties of the nullclines. We have the following result, which we prove in Appendix \ref{app:b}.

\begin{thm}
\label{theorem3}
Consider the reduced two-strain with asymmetric temporary immunity periods system \eqref{SIR-piecewise} and basic reproduction numbers  $\mathscr{R}_1$, $\mathscr{R}_2$, $\mathscr{R}_{12}$, and $\mathscr{R}_{21}$ as in \eqref{r}. Suppose that $\mathscr{R}_1 > 1$, $\mathscr{R}_2 > 1$, and $\mathscr{R}_{21} > 1$ and either $\epsilon = 0$ or $\epsilon = 1$. Then the system has the following properties in \emph{int}$(\Lambda)$:
\begin{enumerate}
    \item \textbf{Vector field:} The $I_2$-nullcline is continuous, has a strictly positive $R_2$-intercept, and is strictly decreasing, and the $R_2$-nullcline is continuous, intercepts the $R_2$-axis at $R_2=0$, and is strictly increasing. Furthermore, the regions bound by the $I_2$- and $R_2$-nullclines have the following directional properties:
    \begin{enumerate}
        \item Above the $I_2$- and $R_2$- nullclines, $I_2' < 0$ and $R_2' < 0$.
        \item Above the $I_2$-nullcline and below the $R_2$-nullcline, $I_2' < 0$ and $R_2' > 0$.
        \item Below the $I_2$-nullcline and above the $R_2$-nullcline, $I_2' > 0$ and $R_2' < 0$.
        \item Below the $I_2$- and $R_2$- nullclines, $I_2' > 0$ and $R_2' > 0$.
    \end{enumerate}
    \item \textbf{Steady state existence and stability:} The $I_2$- and $R_2$-nullclines have a unique intersection, corresponding to a unique positive co-existence steady state $(\bar{I}_2,\bar{R}_2) \in \mbox{\emph{int}}(\Lambda)$. This steady state is locally exponentially stable and the global attractor for trajectories in \emph{int}$(\Lambda)$. Furthermore, we have the following:
    \begin{enumerate}
        \item If $\mathscr{R}_{12} > 1$ then $\displaystyle{0 < \bar{I}_2 + \epsilon \bar{R}_2 < N \left( 1 - \frac{1}{\mathscr{R}_1} \right)}$ and $\bar{I}_1 = \omega(\bar{I}_2,\bar{R}_2) > 0$.
        \item If $\mathscr{R}_{12} < 1$ then $\displaystyle{\bar{I}_2 + \epsilon \bar{R}_2 \geq N \left( 1 - \frac{1}{\mathscr{R}_1} \right)}$ and $\bar{I}_1 = \omega(\bar{I}_2,\bar{R}_2) = 0$.
    \end{enumerate}
\end{enumerate}
\end{thm}

Theorem \ref{theorem3} establishes the uniqueness and global stability of a strictly positive steady state when $\mathscr{R}_1 > 1$, $\mathscr{R}_2> 1$, and $\mathscr{R}_{21} > 1$. It follows that, in these conditions, the emerging strain always has the capacity to infiltrate and survive in a population. Theorem \ref{theorem3} furthermore defines where that steady state lies in the state space. The original strain is able to survive as the emerging strain is infiltrating the population so long as $\mathscr{R}_{12} > 1$ and will die off otherwise. Although Theorem \ref{theorem3} is only proven for the cases of $\epsilon = 0$ (no cross-immunity) and $\epsilon = 1$ (full cross-immunity), we conjecture that the conclusions hold for all values $0 \leq \epsilon \leq 1$.

\section{Numerical Results}
\label{sec:numerical}



In this Section, we conduct some further analysis on the full and reduced two-strain with asymmetric temporary immunity period and partial cross immunity models \eqref{SIR-partial-reduced} and \eqref{SIR-piecewise}.

\begin{figure}[t!]
     \centering
     \begin{subfigure}[h]{0.29\textwidth}
         \centering
         \includegraphics[width=\textwidth]{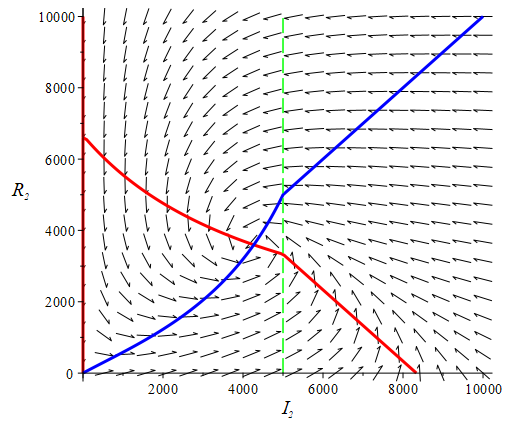}
     \end{subfigure}
     \hfill
     \begin{subfigure}[h]{0.29\textwidth}
         \centering
         \includegraphics[width=\textwidth]{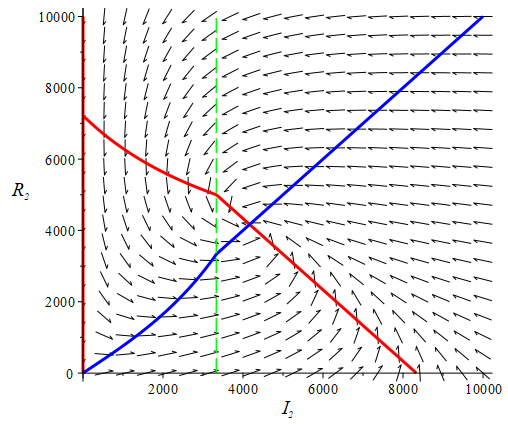}
     \end{subfigure}
     \hfill
     \begin{subfigure}[h]{0.29\textwidth}
         \centering
         \includegraphics[width=\textwidth]{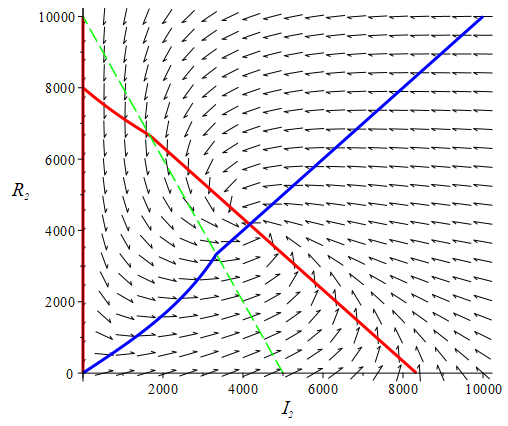}
     \end{subfigure}
     \hfill
     \begin{subfigure}[h]{0.29\textwidth}
         \centering
         \includegraphics[width=\textwidth]{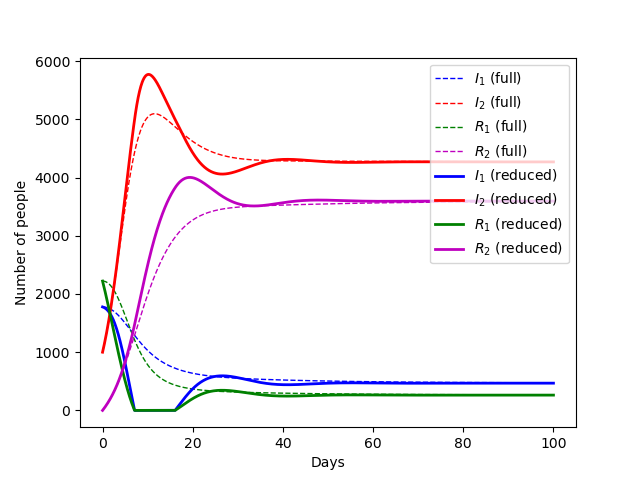}
         \caption{$\beta_1 = 0.4$, $\epsilon = 0$, $\mathscr{R}_1 = 2$, $\mathscr{R}_2 = 6$, $\mathscr{R}_{12} = 1.17$, $\mathscr{R}_{21} = 5$}
         \label{fig:coexistence}
     \end{subfigure}
     \hfill
     \begin{subfigure}[h]{0.29\textwidth}
         \centering
         \includegraphics[width=\textwidth]{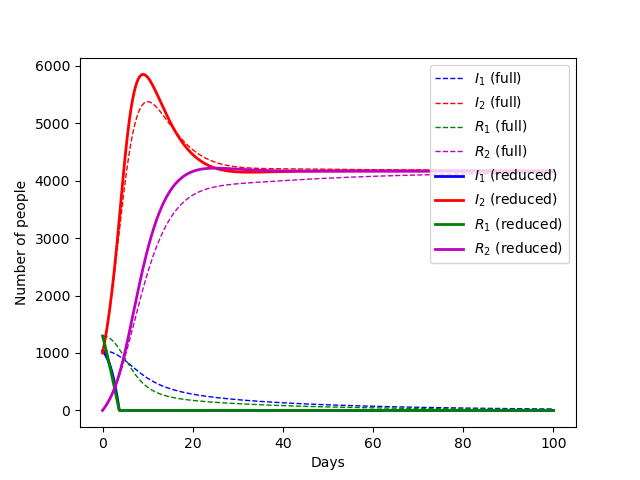}
         \caption{$\beta_1 = 0.3$, $\epsilon = 0$, $\mathscr{R}_1 = 1.5$, $\mathscr{R}_2 = 6$, $\mathscr{R}_{12} = 0.875$, $\mathscr{R}_{21} = 5.33$}
     \end{subfigure}
     \hfill
     \begin{subfigure}[h]{0.29\textwidth}
         \centering
         \includegraphics[width=\textwidth]{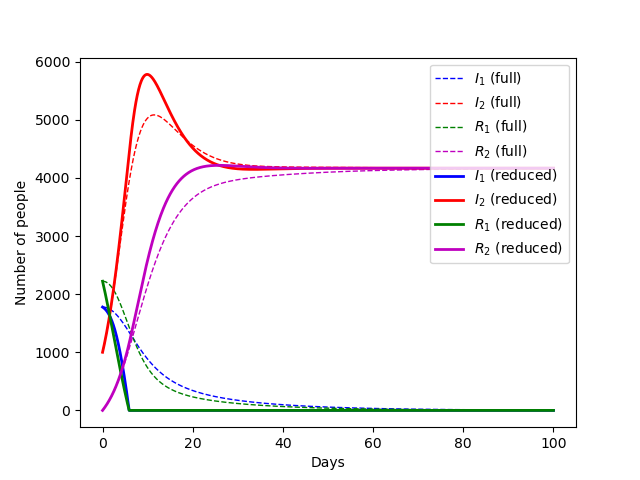}
         \caption{$\beta_1 = 0.4$, $\epsilon = 0.5$, $\mathscr{R}_1 = 2$, $\mathscr{R}_2 = 6$, $\mathscr{R}_{12} = 0.75$, $\mathscr{R}_{21} = 5$}
     \end{subfigure}
        \caption{\small In the upper row, we show vector field plots of reduced model \eqref{SIR-piecewise}, and in the lower row, we show numerical simulations for the full \eqref{SIR-partial-reduced} (dashed) and reduced \eqref{SIR-piecewise} (solid) models. We utilize the parameter values $N = 10000$, $\beta_2 = 0.6$, $\gamma_1 = 0.2$, $\gamma_2 = 0.1$, $\sigma_1 = 0.1$, $\sigma_2 = 0.1$, $I_2(0) = 1000$, and $R_2(0) = 0$. The remaining parameter values are indicated above. For the vector field plots, the $I_2$-nullcline (red), $R_2$-nullcline (blue), and switching line $I_2 + \epsilon R_2 = \displaystyle{N \left( 1 - \frac{1}{\mathscr{R}_1} \right)}$ (dashed green) are indicated. The left of the transition line is the coexistence region ($I_1 > 0$ and $I_2 > 0$) while the right is the competitive exclusion region ($I_1 = 0$ and $I_2 > 0$). Notice that the number of people infected by the original strain ($I_1(t)$) in the reduced system \eqref{SIR-piecewise} may hit zero and then become positive. This occurs in the vector field diagram when a trajectory transitions from the left of the switching line (green) to the right and then back to the left.
        }
        \label{fig:phaseplane}
\end{figure}

\subsection{Vector Field Plots and Numerical Simulations}

We demonstrate the dynamical behavior of the full model \eqref{SIR-DE-reduced} and reduced model \eqref{SIR-piecewise} for three distinct sets of endemic parameters values in Figure \ref{fig:phaseplane}.

The planar dynamics of the reduced system \eqref{SIR-piecewise} may be represented with a vector field plot (Figure \ref{fig:phaseplane}, upper row). We indicate the $I_2$-nullcline (red), $R_2$-nullcline (blue), and switching line $\displaystyle{I_2 + \epsilon R_2 = N \left( 1 - \frac{1}{\mathscr{R}_1} \right)}$ (dashed green). The endemic steady state corresponds to the intersection of the nullclines. When this intersection is to the left of the green line, we have coexistence ($I_1 > 0$ and $I_2 > 0$) while when it is to the right we have emerging strain dominant behavior ($I_1 = 0$ and $I_2 > 0$). In the lower row of Figure \ref{fig:phaseplane}, we include numerical simulations of the full system \eqref{SIR-partial-reduced} (dashed) and the reduced system \eqref{SIR-piecewise} (solid) for the same three sets of parameter values. The initial conditions $I_1(0)$ and $R_2(0)$ for the full system are chosen to be at the endemic steady state value according to \eqref{eq10}.

\subsection{Bifurcation Analysis}

In Figure \ref{fig:bifurcation}, we display selected bifurcation diagrams for the full model \eqref{SIR-partial-reduced} and reduced model \eqref{SIR-piecewise}. In the upper row, we identify four qualitatively distinct regions of behavior in parameter space: \textbf{Region I} - disease-free behavior where neither strain can infiltrate the population; \textbf{Region II} - original strain only behavior where the original strain may infiltrate the population but the emerging strain may not; \textbf{Region III} - emerging strain only behavior where the emerging strain may infiltrate the population but the original strain may not; \textbf{Region IV} - co-existence behavior where both strains may infiltrate the population. In the lower row, we represent how the multi-strain reproductive numbers $\mathscr{R}_{12}$ and $\mathscr{R}_{21}$ change as functions of other selected parameters.

\begin{figure}[t!]
     \centering
     \begin{subfigure}[b]{0.3\textwidth}
         \centering
         \includegraphics[width=\textwidth]{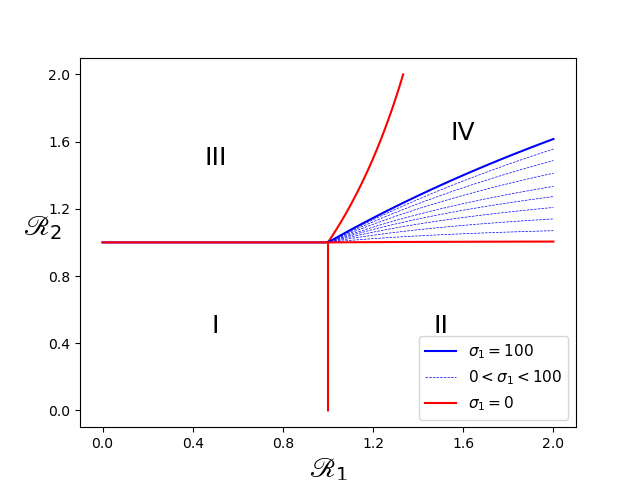}
         \caption{\small $0 \leq \beta_1 \leq 2$, $0 \leq \beta_2 \leq 2$, $\gamma_1 = 1$, $\gamma_2 = 1$, $0 \leq \sigma_1 \leq 100$, $\sigma_2 = 1$, $\epsilon = 0$}
     \end{subfigure}
     \hfill
     \begin{subfigure}[b]{0.3\textwidth}
         \centering
         \includegraphics[width=\textwidth]{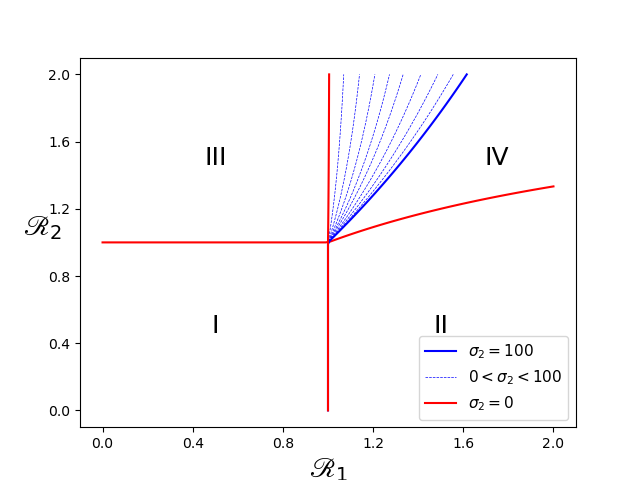}
         \caption{\small $0 \leq \beta_1 \leq 2$, $0 \leq \beta_2 \leq 2$, $\gamma_1 = 1$, $\gamma_2 = 1$, $\sigma_1 = 1$, $0 \leq \sigma_2 \leq 100$, $\epsilon = 0$}
     \end{subfigure}
     \hfill
     \begin{subfigure}[b]{0.3\textwidth}
         \centering
         \includegraphics[width=\textwidth]{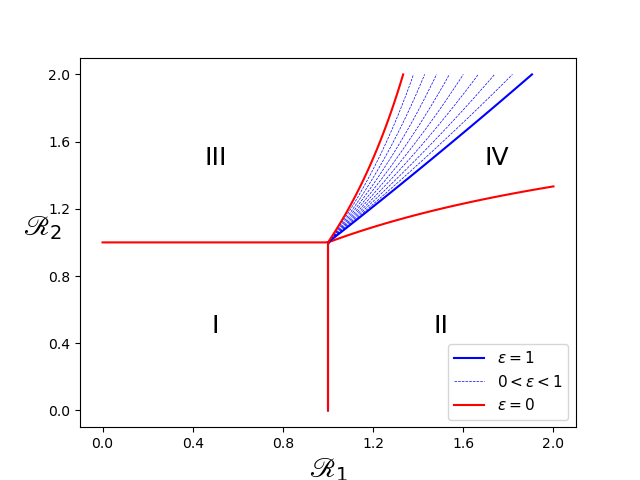}
         \caption{\small $0 \leq \beta_1 \leq 2$, $0 \leq \beta_2 \leq 2$, $\gamma_1 = 1$, $\gamma_2 = 1$, $\sigma_1 \leq 1$, $\sigma_2 = 1$, $0 \leq \epsilon \leq 1$}
     \end{subfigure}
     \hfill
     \begin{subfigure}[b]{0.3\textwidth}
         \centering
         \includegraphics[width=\textwidth]{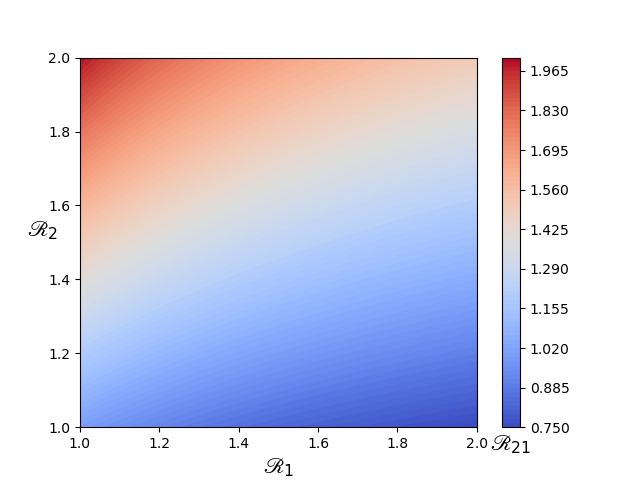}
         \caption{\small $1 \leq \beta_1 \leq 2$, $1 \leq \beta_2 \leq 2$, $\gamma_1 = 1$, $\gamma_2 = 1$, $\sigma_1 = 1$, $\sigma_2 = 1$, $\epsilon = 0.5$}
     \end{subfigure}
     \hfill
     \begin{subfigure}[b]{0.3\textwidth}
         \centering
         \includegraphics[width=\textwidth]{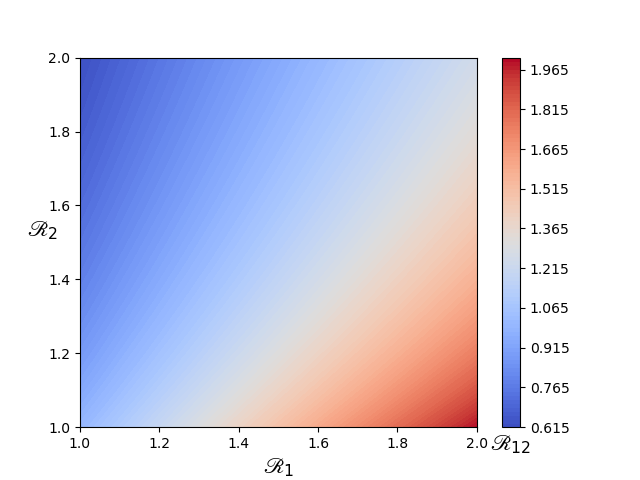}
         \caption{\small $1 \leq \beta_1 \leq 2$, $1 \leq \beta_2 \leq 2$, $\gamma_1 = 1$, $\gamma_2 = 1$, $\sigma_1 = 1$, $\sigma_2 = 1$, $\epsilon = 0.5$}
     \end{subfigure}
     \hfill
     \begin{subfigure}[b]{0.3\textwidth}
         \centering
         \includegraphics[width=\textwidth]{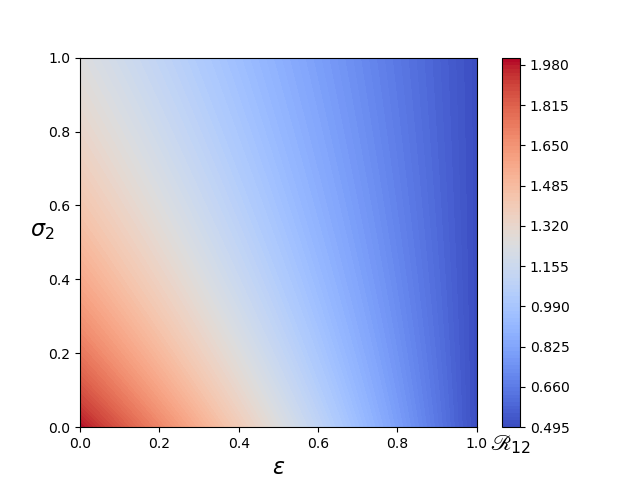}
         \caption{\small $\beta_1 = 4$, $\beta_2 = 2$, $\gamma_1 = 1$, $\gamma_2 = 1$, $\sigma_1 = 1$, $0 \leq \sigma_2 \leq 1$, $0 \leq \epsilon \leq 1$}
     \end{subfigure}
     \hfill
        \caption{\small Bifurcation diagrams for the steady states of \eqref{SIR-piecewise}. In the upper row, the parameter space is divided into four regions corresponding to four qualitatively distinct behaviors: \textbf{Region I} - disease-free behavior where neither strain can infiltrate the population; \textbf{Region II} - original strain only behavior where the original strain may infiltrate the population but the emerging strain may not; \textbf{Region III} - emerging strain only behavior where the emerging strain may infiltrate the population but the original strain may not; \textbf{Region IV} - co-existence behavior where both strains may infiltrate the population. The blue lines indicate how the regions change given changes in the indicated parameter. 
        In the lower row, we display heat maps for the value of $\mathscr{R}_{21}$ (left) and $\mathscr{R}_{12}$ (center and right) as a function of various parameters. Note that the threshold $\mathscr{R}_{21} > 1$ is required for the emerging strain to infiltrate the population and $\mathscr{R}_{12} > 1$ is required for the original strain to survive if the emerging strain infiltrates. The parameters values and ranges are shown above.}
        \label{fig:bifurcation}
\end{figure}


\begin{figure}[t!]
     \centering
     \begin{subfigure}[b]{0.3\textwidth}
         \centering
         \includegraphics[width=\textwidth]{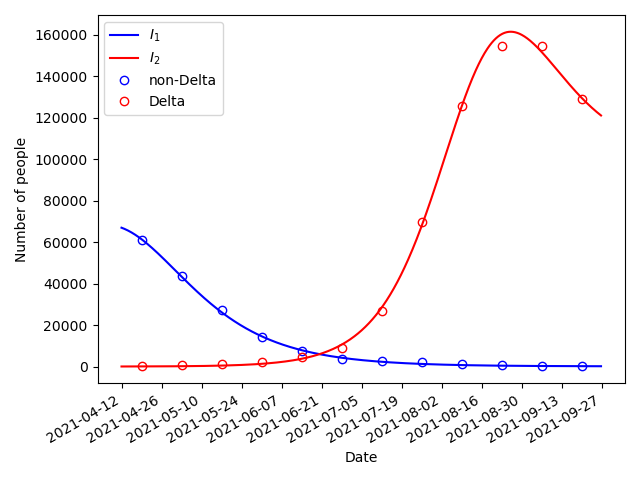}
         \caption{\small Delta (full model)}
         \label{fig:delta-full}
     \end{subfigure}
     \hfill
     \begin{subfigure}[b]{0.3\textwidth}
         \centering
         \includegraphics[width=\textwidth]{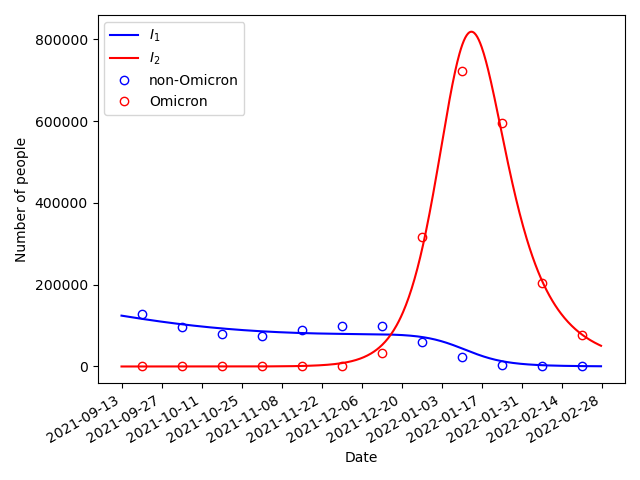}
         \caption{\small Omicron (full model)}
         \label{fig:omicron-full}
     \end{subfigure}
     \hfill
     \begin{subfigure}[b]{0.3\textwidth}
         \centering
         \includegraphics[width=\textwidth]{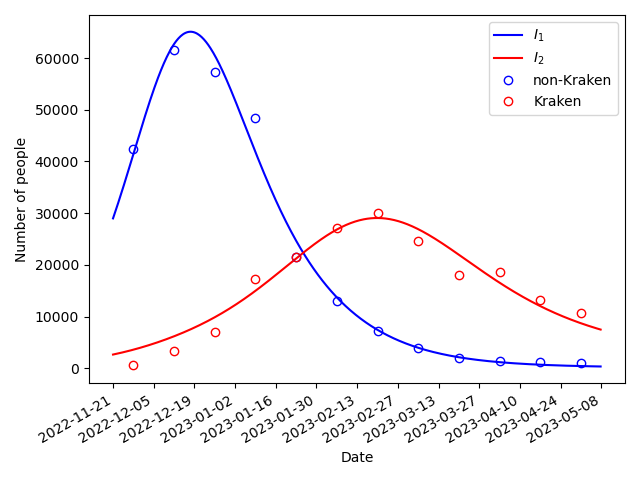}
         \caption{\small Kraken (full model)}
         \label{fig:xbb-full}
     \end{subfigure}
     \hfill
     \begin{subfigure}[b]{0.3\textwidth}
         \centering
         \includegraphics[width=\textwidth]{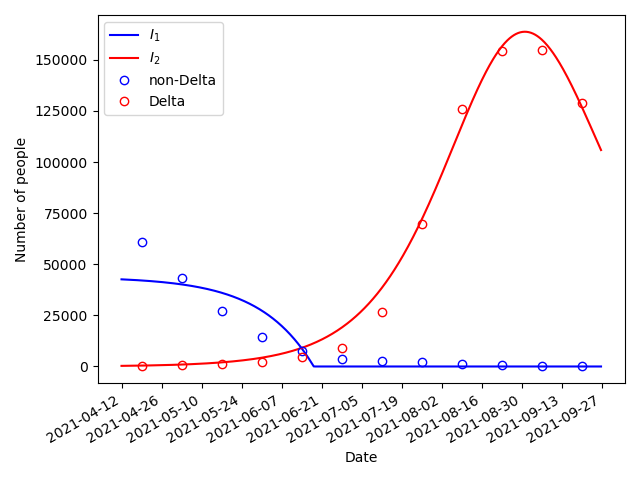}
         \caption{\small Delta (reduced model)}
         \label{fig:delta-reduced}
     \end{subfigure}
     \hfill
     \begin{subfigure}[b]{0.3\textwidth}
         \centering
         \includegraphics[width=\textwidth]{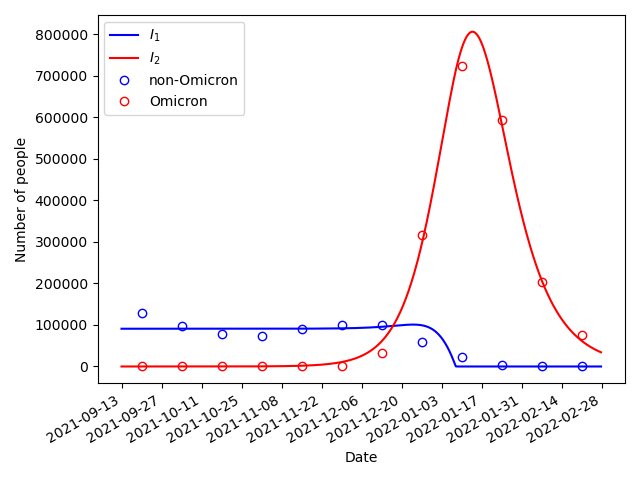}
         \caption{\small Omicron (reduced model)}
         \label{fig:omicron-reduced}
     \end{subfigure}
     \hfill
     \begin{subfigure}[b]{0.3\textwidth}
         \centering
         \includegraphics[width=\textwidth]{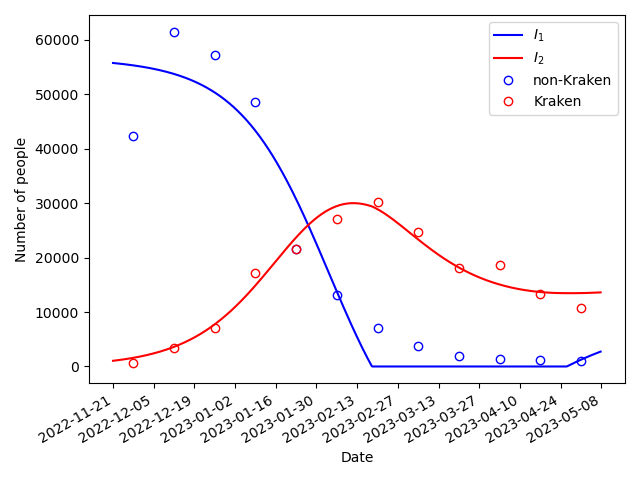}
         \caption{\small Kraken (reduced model)}
         \label{fig:xbb-reduced}
     \end{subfigure}
     \hfill
        \caption{\small Comparisons of COVID-19 case incidence data to the best-fitting full model \eqref{SIR-partial-reduced} and reduced model \eqref{SIR-piecewise}. In all figures, $I_2$ denotes the number of new cases of the variant of interest aggregated over the preceding two week period [(\textbf{(a)} \& \textbf{(d)}: Delta (B.1.617.2); \textbf{(b)} \& \textbf{(e)}: Omicron (B.1.1.529); \textbf{(c)} \& \textbf{(f)}: Omicron-subvariant Kraken (XBB.1.5)] and $I_1$ denotes the number of new cases over all other COVID-19 strains. Parameters are fit using the sum of squared error formula \eqref{error} in Python. COVID-19 case data is taken from the John Hopkins \cite{dong2023} and Our World in Data \cite{owidcoronavirus} and the variant data is taken from GISAID \cite{shu2017gisaid}. Data points are displayed at the midpoint of the two-week period during which they are aggregated. Best fitting parameters values can be found in Table \ref{table2}.}
        \label{fig:datafitting}
\end{figure}

\begin{table}[t!]
    \centering
    \small
    \begin{tabular}{c||r|r||r|r||r|r|}
    Parameters & \multicolumn{2}{c||}{Delta} & \multicolumn{2}{c||}{Omicron} & \multicolumn{2}{c}{Kraken} \\
\hline \hline
     & full \textbf{(a)} & reduced \textbf{(d)} & full \textbf{(b)} & reduced \textbf{(e)} & full \textbf{(c)} & reduced \textbf{(f)} \\
    \cline{2-7} 
    $N$ & 16336307 & 48040094 & 31579543 & 45857562 & 7575545 & 2993020 \\
    $I_1(0)$ & 188363 & 188363 & 1462667 & 813673 & 96790 & 283535 \\
    $R_1(0)$ & 308476 & 308476 & 15725210 & 12 & 0 & 100 \\
    $I_2(0)$ & 85 & 722 & 1 & 2 & 9388 & 3661 \\
    $R_2(0)$ & 85 & 0 & 0 & 0 & 0 & 3661 \\
    $\beta_1$ & 0.36685 & 0.36316 & 0.18669 & 0.29794 & 0.30167 & 0.31235 \\
    $\beta_2$ & 0.43498 & 0.41680 & 0.23336 & 0.37243 & 0.28631 & 0.30993 \\
    $\gamma_1$ & 0.35878 & 0.36162 & 0.09294 & 0.24258 & 0.23641 & 0.22252 \\
    $\gamma_2$ & 0.35878 & 0.36162 & 0.09294 & 0.24258 & 0.23641 & 0.22252 \\
    $\sigma_1$ & 0.00266 & 0.5 & 0.00676 & 0.5 & 0.01421 & 0.09208 \\
    $\sigma_2$ & 0.04600 & 0.001 & 0.001 & 0.001 & 0.00544  & 0.02175 \\
    $\epsilon$ & 0.001 & 1 & 1 & 1 & 1 & 1 \\
    $\mathscr{R}_1$ & 1.02250 & 1.00426 & 2.00875 & 1.22823 & 1.27603 & 1.40369 \\
    $\mathscr{R}_2$ & 1.21237 & 1.15260 & 2.51093 & 1.53528 & 1.21107 & 1.39284 \\
    $\mathscr{R}_{12}$ & 1.00199 & 0.87130 & 0.80000 & 0.80000 & 1.05364 & 1.00779 \\
    $\mathscr{R}_{21}$ & 1.21217 & 1.14977 & 2.42534 & 1.52273 & 1.19621 & 1.27560 \\
    \hline \hline
    \end{tabular}
    \caption{\small Best-fitting parameter values for the model simulations shown in Figure \ref{fig:datafitting}. For the full models \eqref{SIR-partial-reduced} \textbf{(a)}-\textbf{(c)}, we fit the parameters $N$, $I_1(0)$, $R_1(0)$, $I_2(0)$, $R_2(0)$, $\beta_1$, $\beta_2$, $\gamma_1$, $\gamma_2$, $\sigma_1$, $\sigma_2$, and $\epsilon$. For the reduced model \eqref{SIR-piecewise} \textbf{(d)}-\textbf{(f)}, we fit the parameters $N$, $I_2(0)$, $R_2(0)$, $\beta_1$, $\beta_2$, $\gamma_1$, $\gamma_2$, $\sigma_1$, $\sigma_2$, and $\epsilon$ and determine the values of $I_1(0)$ and $R_1(0)$ from the system of equations \eqref{eq10}. In order to focus on the potential impact of the temporary immunity periods $\sigma_1^{-1}$ and $\sigma_2^{-1}$ and the degree of cross-immunity $\epsilon$, we imposed the restrictions $0.8 \beta_1 \leq \beta_2 \leq 1.25 \beta_1$, $\gamma_1 = \gamma_2$, $\sigma_1 \leq 0.5$, and $\sigma_2 \leq 0.5$. We also impose that all parameters are greater than or equal to $0.001$.}
    \label{table2}
\end{table}

\subsection{Data Fitting Methods}

To validate the differential equation models \eqref{SIR-partial-reduced} and \eqref{SIR-piecewise}, we fit them to COVID-19 incidence and variant proportion data from the United States. COVID-19 case incidence data prior to 10/10/2022 was retrieved from the Johns Hopkins GitHub repository on May 5, 2023 \cite{dong2023} and after 10/10/2022 was retrieved from Our World in Data on June 5, 2023 \cite{owidcoronavirus}. Variant proportions data was retrieved from the Global Initiative on Sharing All Influenza Data (GISAID) on June 5, 2023 \cite{shu2017gisaid}.

We consider three separate periods of the COVID-19 pandemic during which they were a shift from the dominance of one variant to another in the United States: (1) the Delta takeover period (04/07/2021-08/16/2021); (2) the Omicron takeover period (08/30/2021-02/14/2022); and (3) the Omicron subvariant Kraken takeover period (11/21/2022-05/08/2023). In all cases, we consider the variant which is taking over as the emerging strain, $I_2$, and aggregate all remaining strains as the original strain, $I_1$. To determine variant numbers, we take the raw COVID-19 incidence data aggregated over the preceding two-week period and divide it according to the proportion of cases that are the emerging strain and which are the original strain. For optimizing the fit of the data to the model, we use the sum of squared error function
\begin{equation}
\label{error}
E(\mathbf{x}, \mathbf{y};\theta) = \sum_{i=1}^n (x_i - \hat{x}_i)^2 + \sum_{i=1}^n (y_i - \hat{y}_i)^2.
\end{equation}
The values $(x_i,y_i)$ are the cumulative clinically-confirmed COVID-19 cases over the preceding two week period for the original ($x$) and emerging ($y$) strain, respectively, in the $i^{th}$ period of data collection. The values $(\hat{x}_i,\hat{y}_i)$ are biweekly increase in new cases over the same time periods derived from the model. To simulate the model trajectories, we use the 4th order Runge-Kutta method. To fit the models to data, we use a customized stochastic Newton's method algorithm written in Python. We note that other COVID-19 studies have fit to variant proportions rather than estimated case numbers \cite{boyle2022,ciupeanu2022}.

In order to highlight the effects of the asymmetric temporary immunity periods and partial cross-immunity, we restrict the variant-specific transmission rates $\beta_i$ and recovery rates $\gamma_i$ to the limited ranges $0.8 \beta_1 \leq \beta_2 \leq 1.25 \beta_1$ and $\gamma_1 = \gamma_2$. That is, we assume that the emerging strain is no more or less than $25\%$ as transmissible as the original strain and that the two strains has the same recovery period. Otherwise, we restrict all parameters from going to zero and $\sigma_i < 0.5$. For the full model fits \eqref{SIR-partial-reduced}, we fit over the parameters $N$, $I_1(0)$, $R_1(0)$, $I_2(0)$, $R_2(0)$, $\beta_1$, $\beta_2$, $\gamma_1$, $\gamma_2$, $\sigma_1$, $\sigma_2$, and $\epsilon$. For the reduced model fits \eqref{SIR-piecewise}, we fit over the parameters $N$, $I_1(0)$, $R_1(0)$, $I_2(0)$, $R_2(0)$, $\beta_1$, $\beta_2$, $\gamma_1$, $\gamma_2$, $\sigma_1$, $\sigma_2$, and $\epsilon$. Note that we also fit the population size $N$. This may seem unusual since the overall population size of the United States is well-known. The choice of fitting $N$ was made to account for several factors which are not considered in the model: (1) removal of individuals from the susceptible population due to natural immunity, quarantine, and vaccination; and (2) the incompleteness of COVID-19 case incidence data due to reporting errors and the prevalence of at-home self-test kits. The value $N$ can be more fairly interpreted as the effective population level after accounting for the aggregate effects of removal from susceptibility and underreporting.

The best fitting model simulations are shown in Figure \ref{fig:datafitting} and the corresponding best fitting parameters can be found in Table \ref{table2}. The data circle corresponds to the average clinically-confirmed daily incidence of COVID-19 over the surrounding two week period. For ease of visualization, the data circle is displayed at the midpoint of the two-week data collection period.

\section{Discussion}

In this section, we analyze the results of the numerical studies carried out in Section \ref{sec:numerical}.


\subsection{Full and Reduced Model Comparison}

Consider the numerical simulations of the full model \eqref{SIR-partial-reduced} and reduced model \eqref{SIR-piecewise} (lower row of Figure \ref{fig:phaseplane}). Note that there is no significant time-scale separation between the dynamics of the original and emerging strain so that the assumptions underlying the QSSA are not met. Consequently, we should not expect that trajectories of \eqref{SIR-partial-reduced} converge to those of \eqref{SIR-piecewise} in any limiting way. Nevertheless, our numerical study suggests that trajectories of the full and reduced systems agree well when the full system is taken to start near the endemic steady state of the original strain. After a short transient period, trajectories of both systems settle asymptotically to the same steady state. Since the long-term dynamics agree, we conjecture that the global stability of the steady states guaranteed by Theorem \ref{theorem3} holds not only for the reduced system \eqref{SIR-piecewise} but also for the full system \eqref{SIR-partial-reduced}.

We now consider the performance of the full \eqref{SIR-partial-reduced} and reduced \eqref{SIR-piecewise} models when fit with COVID-19 case incidence data (see Figure \ref{fig:datafitting}).
\begin{enumerate}
\item
During the Delta takeover period (Figure \ref{fig:datafitting}\textbf{(a)} and \textbf{(d)}), the full model fits the declining trend of the original strain ($I_1$) noticeably better than the reduced model. We also note a disparity in the survival prediction for the original strains ($\mathscr{R}_{12} > 1$ for full model and $\mathscr{R}_{12} < 1$ for the reduced model). This disparity can be attributed to a violation of the assumption underlying the derivation of the reduced model, which was that the original strain has reached its endemic steady state before the infiltration of the emerging strain.
\item
During the Omicron takeover period (Figure \ref{fig:datafitting}\textbf{(b)} and \textbf{(e)}), the two models fit comparably well. In particular, the two models consistently predict that Omicron was more transmissible than previous strains ($\mathscr{R}_2 > \mathscr{R}_1$), that Omicron would infiltrate the population ($\mathscr{R}_{21} > 1$), and that the previous strains would not survive ($\mathscr{R}_{12} < 1$).
\item
During the Kraken takeover period (Figure \ref{fig:datafitting}\textbf{(c)} and \textbf{(f)}), the two models fit the emerging strain ($I_2$) comparably well, while the reduced model fits the pre-takeover phase of the original strain ($I_1$) better and the full model fits the post-takeover phase of the original strain better. In particular, the reduced model better captures the steadiness of the original strain pre-takeover but overshoots its decline as the emerging strain enters the population. Both models, however, consistently predict the infiltration of the Kraken ($\mathscr{R}_{21} > 1$) and survival of the original strain ($\mathscr{R}_{12} >1$).
\end{enumerate}
Overall, the reduced model \eqref{SIR-piecewise} performs well when fit with COVID-19 case incidence data. In particular, it fits the dynamics of the emerging strain ($I_2$) comparably well to the full mode. Care should be taken, however, to ensure the model assumption that the prevalence of existing strains is steady prior to the emergence of the new strain is met.

\subsection{Effect of Temporary Immunity and Partial Cross-Immunity}

Now consider the bifurcation plots in Figure \ref{fig:bifurcation}. The plots \textbf{(b)} and \textbf{(c)} show how changes in the loss of temporary immunity rates, $\sigma_1$ and $\sigma_2$, influence the boundaries of the four qualitatively distinct regions of parameter space. Note that as the temporary immunity periods become shorter (i.e. $\sigma_i \to \infty$, $\sigma_i^{-1} \to 0$), the co-existence region becomes smaller and the region for one-strain dominant behavior grows. This decrease in the capacity of coexistence as the temporary immunity periods decrease can be intuitively attributed to the increased amount of time where the two strains are competing for infections from a common susceptible class ($S$) rather than asymmetric pools of susceptibles ($S+R_1$ or $S+R_2$). In this case, it is more likely that the strictly more contagious strain will survive ($\mathscr{R}_1 > \mathscr{R}_2$ or $\mathscr{R}_2 > \mathscr{R}_1$). This analysis suggests that asymmetric temporary immunity periods is an important factor in the long-term survival of individual strains of COVID-19. Specifically, it is important for individual strains to be able to infect those recently infected with other strains to gain an advantage over those strains. 

The plot \textbf{(c)} shows how changes in the degree of cross-immunity from the emerging strain to the original strain, $\epsilon$, influences the boundaries of the regions. We can see that as the degree of cross-immunity increases ($\epsilon \to 1$), the region for coexistence shrinks dramatically. In this case, the original strain needs to have a strictly higher basic reproduction number in order to survive ($\mathscr{R}_1 > \mathscr{R}_2$). This can be intuitively interpreted as resulting from that reduction in individuals susceptible to the original strain ($S$) relative to the emerging strain ($S+R_1$). This suggests that even partial cross-immunity from the emerging strain to the original strain can contribute a significant competitive advantage to the emerging strain. Notably, an emerging strain of COVID-19 need not be significantly more contagious than an original strain (i.e. $\mathscr{R}_2 \gg \mathscr{R}_1$) in order for the new strain to eliminate the original strain from the population.

The plot \textbf{(d)} displays how the basic reproduction number for the emerging strain in the presence of the original strain $\mathscr{R}_{21}$ changes as a function of the strain-specific basic reproduction numbers $\mathscr{R}_1$ and $\mathscr{R}_2$. Having a more contagious emerging strain ($\mathscr{R}_2$ high) and a less contagious original strain ($\mathscr{R}_1$ low) makes it more likely for emerging strain to be able to infiltrate the population ($\mathscr{R}_{21} > 1$). The plots \textbf{(e)} and \textbf{(f)} display how the basic reproduction number for the original strain in the presence of the emerging strain $\mathscr{R}_{12}$ changes as a function of $\mathscr{R}_1$, $\mathscr{R}_2$, $\sigma_2$, and $\epsilon$. A more contagious original strain ($\mathscr{R}_1$ high), less contagious emerging strain ($\mathscr{R}_2$ low), long temporary immunity period for the emerging strain ($\sigma_2^{-1}$ high), or low degree of cross-immunity ($\epsilon$ low) makes it more likely for the original strain to survive when the emerging strain infiltrates the population ($\mathscr{R}_{12} > 1$). The analysis suggests that the ability of COVID-19 strains to exclusively infect those recently infected by emerging strains is crucial for their long-term survival.

\subsection{Implications for the spread of COVID-19}



Notice that the parameter-fit strain-specific basic reproduction numbers $\mathscr{R}_1$ and $\mathscr{R}_2$ are not significantly different in any of the models (see Table \ref{table2}). In fact, for the Kraken the models predict that the emerging strain is slightly less transmissible than the original strain ($\mathscr{R}_1 > \mathscr{R}_2$). This suggests that factors other than differences in transmissibility can be a driving force for the takeover of an emerging strain. For cases \textbf{(b)}-\textbf{(f)} in Figure \ref{fig:datafitting} we have that the best fitting parameters include $\epsilon = 1$ while for case \textbf{(a)} we have $\sigma_1^{-1} > \sigma_2^{-1}$. This suggests that the emerging strain can gain a competitive advantage by increasing its cross-immunity or shortening its temporary immunity period. In these cases, even though the two strains have comparable strain-specific reproduction numbers, the emerging strain is able to infiltrate the already-infected population and drive the original strain to extinction or near-extinction.

For both the full and the reduced models, the data fitting suggests that the highest capacity of strain-over-strain fitness is for Omicron over previous strains ($\mathscr{R}_{21} = 2.42534$ for full model and $\mathscr{R}_{21} = 1.52273$ for reduced model). This is consistent with the data trend which showed Omicron accounting for over $99\%$ of COVID-19 cases in the United States within three months of first being detected \cite{shu2017gisaid}. The data fitting also suggests that coexistence is possible with the Kraken strain and previous strains ($\mathscr{R}_{12} = 1.03467 > 1$ for full model and $\mathscr{R}_{21} = 1.10030 >1$ for reduced model). Even though we have seen the prevalence of Kraken increase dramatically over the indicated time period, the model fits suggests the ancestral strains will survive rather than die off like previous takeovers. This is consistent with the observation that the Kraken variant has comprised approximate $90\%$ of COVID-19 cases during post-takeover period of 03/27/2023-06/05/2023 \cite{shu2017gisaid}. This suggests that we are closer to having the long-term reality of co-circulating COVID-19 strains. Such a situation requires an adjustment of public health strategies to predict the most effective strain-specific vaccines, as is currently conducted with seasonal influenza.

\section{Conclusions and Future Work}

We have introduced a two-strain model for infectious disease spread which incorporates asymmetric temporary immunity periods and partial cross-immunity. We have derived conditions (Theorem \ref{theorem1}) on the temporary immunity periods, degree of cross-immunity, and basic reproduction numbers under which the strains can coexist. We have also reduced the model to a planar hybrid switching system and analysed the dynamics using linear stability analysis, phase plane analysis, and the Bendixson-Dulac criterion. We have parameter fit the full model \eqref{SIR-partial-reduced} and reduced model \eqref{SIR-piecewise} to COVID-19 case incidence data from the United States for three different strains. This analysis has demonstrated the capacity of differences in temporary immunity periods and partial cross-immunity to account for the observed changes in variant proportions over time, and in particular the phenomenon of one variant infiltrating a population and eliminating previous variants.

The work conducted in this paper suggests several fruitful opportunities for future work:
\begin{enumerate}
    \item 
    \emph{Incorporating vaccination.} The full and reduced models \eqref{SIR-partial-reduced} and \eqref{SIR-piecewise} do not incorporate vaccination. Data has suggested, however, that vaccination has played a significant role in altering the spread of COVID-19 \cite{johnston2022}. Future work will incorporate vaccination in the models to the impact of asymmetries in vaccine-resistance between strains in promoting competitive exclusion or coexistence.
    \item 
    \emph{Extending to more strains.} The analysis in this paper has focused on the two-strain models \eqref{SIR-partial-reduced} and \eqref{SIR-piecewise}. Throughout the COVID-19 pandemic, however, we have periods with more than two circulating strains, and the impact of multiple circulating strains is expected to become more pronounced as COVID-19 becomes endemic in the global population. Extending the model set-up, bifurcation analysis, and reduction methods to models with more than two strains will be the focus of future work.
    \item 
    \emph{Extending dynamical results to the full model \eqref{SIR-partial-reduced} and the reduced model \eqref{SIR-piecewise} for $0 \leq \epsilon \leq 1$.} It is suspected that the conclusions of Theorem \ref{theorem3} hold for the reduced model \eqref{SIR-DE-reduced} for all values of $0 \leq \epsilon \leq 1$, rather than just the special cases $\epsilon = 0$ and $\epsilon = 1$, and also the full model \eqref{SIR-partial-reduced}. These results, however, remain unproved. A sufficient condition for Theorem \ref{theorem3} to hold for the reduced model \eqref{SIR-piecewise} would be
    \[\frac{\partial \omega}{\partial R_2} R + \omega(I_2,R_2) \geq 0\]
    for $\displaystyle{0 < I_2 + \epsilon R_2 < N \left( 1 - \frac{1}{\mathscr{R}_1} \right)}$. Future work will aim to complete this analysis.
    \item
    \emph{Conducting asymptotic analysis for reduction of full model \eqref{SIR-partial-reduced} to reduced model \eqref{SIR-piecewise}.} Although model simulations agree well in the long-term dynamics (see Figure \ref{fig:phaseplane}), evaluating the convergence of trajectories in the transient dynamics is made challenging by the cross-compartmental infection terms. Determining parameter regions where trajectory convergence is guaranteed will be a focus of future research.
    \item
    \emph{Incorporating births and deaths.} Since we are considering strains with emerge within a population with a circulating endemic strain, it is natural to incorporate long-term population effects such as births and deaths. This will be the focus of future research.
\end{enumerate}

\paragraph*{Funding}Work on this project was supported by NSF Grant DMS-2213390.

\paragraph*{Acknowledgments}We gratefully acknowledge all data contributors, i.e., the Authors and their Originating laboratories responsible for obtaining the specimens, and their Submitting laboratories for generating the genetic sequence and metadata and sharing via the GISAID Initiative, on which this research is based.

\paragraph*{Conflicts of interest}The authors declare no conflict of interest.


\appendix

\section{Next Generation Method}
\label{app:ngm}

Following \cite{VANDENDRIESSCHE2002}, we define the state vector $\mathbf{x} \in \mathbb{R}_{\geq 0}^n$ and reindex so that $i= 1, \ldots, m$, $m \leq n$, corresponds to the infectious compartments. For $i=1, \ldots, m,$ let $\mathscr{F}_i(\mathbf{x})$ denote the rate of new infections in compartment $i$, $\mathscr{V}^+_i(\mathbf{x})$ denote the rate of inflow into compartment $i$ by any other means, and $\mathscr{V}^-_i(\mathbf{x})$ denote the rate of outflow from compartment $i$ by any other means. The rate of change in each infectious compartment $i = 1, \ldots, m,$ can then be given by
\[\frac{dx_i}{dt} = \mathscr{F}_i - \mathscr{V}_i\]
where $\mathscr{V}_i = \mathscr{V}^-_i - \mathscr{V}^+_i$. Let
\[F = \left[ \frac{d\mathscr{F}_i}{dx_j} \right], \; i, j = 1, \ldots, m \]
and
\[V = \left[ \frac{d\mathscr{V}_i}{dx_j} \right], \; i, j = 1, \ldots, m \]
denote the Jacobians of $\mathscr{F}$ and $\mathscr{V}$, respectively, restricted to only infectious compartments. The next generation matrix is given by $FV^{-1}$. If $\mathbf{x}_0$ is the disease free steady state, the basic reproduction number is given by $\mathscr{R}_0 = \rho(FV^{-1}(\mathbf{x}_0))$. If $FV^{-1}$ is reducible with $n$ irreducible components, we can decouple the system into strain specific submatrices $F_i V_i^{-1}$. We can then compute the basic reproduction number of strain $i$ by $\mathscr{R}_i = \rho(F_i V_i^{-1}(\mathbf{x}_0))$ and derive $\displaystyle{\mathscr{R}_0 = \max_{i = 1, \ldots, m } \{ \mathscr{R}_i \}}$. Similarly, if $\mathbf{x}_j$ is the strain $j$ only steady state, then the basic reproduction number of strain $i$ in the presence of strain $j$ is given by $\mathscr{R}_{ij} = \rho(F_iV_i^{-1}(\mathbf{x}_j))$.

\subsection{Basic two-strain model}
\label{app:basic}

For the basic two-strain model \eqref{SIR}, we have the following:
\[\mathscr{F} = \left[ \begin{array}{c}
\displaystyle{\frac{\beta_1}{N} S I_1} \\[0.1in]
\displaystyle{\frac{\beta_2}{N} S I_2}
\end{array} \right] \mbox{ and } \mathscr{V} = \left[ \begin{array}{c}
\gamma_1 I_1 \\[0.1in]
\gamma_2 I_2 
\end{array} \right] \]
so that
\[F = \left[ \begin{array}{cc}
\displaystyle{\frac{\beta_1}{N} S} & 0 \\
0 & \displaystyle{\frac{\beta_2}{N} S}
\end{array} \right] \mbox{ and } V = \left[ \begin{array}{cc}
\gamma_1 & 0 \\
0 & \gamma_2
\end{array} \right].\]
It follows that
\[FV^{-1} = \left[ \begin{array}{cc}
\displaystyle{\frac{\beta_1}{\gamma_1 N} S} & 0 \\
0 & \displaystyle{\frac{\beta_2}{\gamma_2 N} S}
\end{array} \right], F_1V_1^{-1} = \displaystyle{\frac{\beta_1}{\gamma_1 N} S} , \mbox{ and } F_2V_2^{-1} =\displaystyle{\frac{\beta_2}{\gamma_2 N} S}\]
where we have removed the matrix representation for the $1 \times 1$ matrices for notational simplicity. Substituting in the disease free steady state $\mathbf{x}_0$ \eqref{dfe} gives
\[FV^{-1}(\mathbf{x}_0) = \left[ \begin{array}{cc}
\displaystyle{\frac{\beta_1}{\gamma_1} } & 0 \\
0 & \displaystyle{\frac{\beta_2}{\gamma_2}}
\end{array} \right], F_1V_1^{-1}(\mathbf{x}_0) = \displaystyle{\frac{\beta_1}{\gamma_1} }, \mbox{ and } F_2V_2^{-1}(\mathbf{x}_0) =\displaystyle{\frac{\beta_2}{\gamma_2}}. \]
It follows that $\mathscr{R}_1 = \beta_1/\gamma_1$, $\mathscr{R}_2 = \beta_2/\gamma_2$, and $\mathscr{R}_0 = \max \{ \mathscr{R}_1, \mathscr{R}_2 \}$. Furthermore, substituting in the original strain only steady state $\mathbf{x}_1$ \eqref{ose} and emerging strain only steady state $\mathbf{x}_2$ \eqref{ese} into the relevant components gives
\[F_1V_1^{-1}(\mathbf{x}_2) = \displaystyle{\frac{\beta_1 \gamma_2}{\beta_2 \gamma_1} }, \mbox{ and } F_2V_2^{-1}(\mathbf{x}_1) =\displaystyle{\frac{\beta_2 \gamma_1}{\beta_1 \gamma_2}}.\]
It follows that $\mathscr{R}_{12} = \mathscr{R}_1 / \mathscr{R}_2$ and $\mathscr{R}_{21} = \mathscr{R}_2 / \mathscr{R}_1$ and we are done.

\subsection{Two-strain model with asymmetric temporary immunity periods and partial cross-immunity}
\label{app:new}

For the two-strain model with asymmetric temporary immunity periods and partial cross-immunity \eqref{SIR2}, we have the following:
\[\mathscr{F} = \left[ \begin{array}{c}
\displaystyle{\frac{\beta_1}{N} (S + (1- \epsilon)R_2) I_1} \\[0.1in]
\displaystyle{\frac{\beta_2}{N} (S + R_1) I_2}
\end{array} \right] \mbox{ and } \mathscr{V} = \left[ \begin{array}{c}
\gamma_1 I_1 \\[0.1in]
\gamma_2 I_2 
\end{array} \right] \]
so that
\[F = \left[ \begin{array}{cc}
\displaystyle{\frac{\beta_1}{N} (S + (1-\epsilon)R_2)} & 0 \\
0 & \displaystyle{\frac{\beta_2}{N} (S + R_1)}
\end{array} \right] \mbox{ and } V = \left[ \begin{array}{cc}
\gamma_1 & 0 \\
0 & \gamma_2
\end{array} \right].\]
It follows that
\[\small FV^{-1} = \left[ \begin{array}{cc}
\displaystyle{\frac{\beta_1}{\gamma_1 N} (S + (1-\epsilon)R_2)} & 0 \\
0 & \displaystyle{\frac{\beta_2}{\gamma_2 N} (S+R_1)}
\end{array} \right], F_1V_1^{-1} = \displaystyle{\frac{\beta_1}{\gamma_1 N} (S + (1-\epsilon)R_2)} , \mbox{ and } F_2V_2^{-1} =\displaystyle{\frac{\beta_2}{\gamma_2 N} (S+R_1)}. \]
Substituting in the disease free steady state $\mathbf{x}_0$ \eqref{dfe} gives
\[FV^{-1}(\mathbf{x}_0) = \left[ \begin{array}{cc}
\displaystyle{\frac{\beta_1}{\gamma_1} } & 0 \\
0 & \displaystyle{\frac{\beta_2}{\gamma_2}}
\end{array} \right], F_1V_1^{-1}(\mathbf{x}_0) = \displaystyle{\frac{\beta_1}{\gamma_1} }, \mbox{ and } F_2V_2^{-1}(\mathbf{x}_0) =\displaystyle{\frac{\beta_2}{\gamma_2}}. \]
It follows that $\mathscr{R}_1 = \beta_1/\gamma_1$, $\mathscr{R}_2 = \beta_2/\gamma_2$, and $\mathscr{R}_0 = \max \{ \mathscr{R}_1, \mathscr{R}_2 \}$. Now, substituting in the original strain only steady state $\mathbf{x}_1$ \eqref{ose} and emerging strain only steady state $\mathbf{x}_2$ \eqref{ese} into the relevant components gives
\[F_1V_1^{-1}(\mathbf{x}_2) = \displaystyle{\frac{\beta_1 \gamma_2((1-\epsilon)\beta_2 + \epsilon \gamma_2 + \sigma_2)}{\beta_2 \gamma_1 ( \gamma_2 + \sigma_2 )} }, \mbox{ and } F_2V_2^{-1}(\mathbf{x}_1) =\displaystyle{\frac{\beta_2 \gamma_1(\beta_1 + \sigma_1)}{\beta_1 \gamma_2( \gamma_1 + \sigma_1)}}.\]
It follows that $\mathscr{R}_{12} = \displaystyle{\frac{\mathscr{R}_1}{\mathscr{R}_2} \left( \frac{(1-\epsilon)\beta_2 + \epsilon \gamma_2 + \sigma_2}{\gamma_2 + \sigma_2} \right)}$ and $\mathscr{R}_{21} = \displaystyle{ \frac{\mathscr{R}_2}{\mathscr{R}_1} \left(\frac{\beta_1 + \sigma_1}{\gamma_1 + \sigma_1}\right)}$ and we are done.

\section{Proof of Theorem \ref{thm-coexistence}}
\label{app:a}

\begin{proof}
We want to prove that a co-existence steady state (i.e. $I_1 >0$ and $I_2 > 0$) exists if and only if $\min \{ \mathscr{R}_1, \mathscr{R}_2, \mathscr{R}_{12}, \mathscr{R}_{21} \} > 1$ is satisfied. 

We consider the steady state equations for \eqref{SIR-partial-reduced}:
\begin{equation}
    \label{steadystate}
    \left\{ \; \; \;
    \begin{aligned}
\frac{\beta_1}{N} (N - I_1 - R_1 - I_2 - \epsilon R_2) I_1 - \gamma_1 I_1 &  = 0\\
\gamma_1 I_1 - \sigma_1 R_1 - \frac{\beta_2}{N} R_1 I_2 & = 0\\
\frac{\beta_2}{N} (N - I_1 - I_2 - R_2) I - \gamma_2 I_2 & = 0\\
\gamma_2 I_2 - \sigma_2 R_2 - \frac{\beta_1}{N} (1 - \epsilon) R I_1 & = 0
    \end{aligned}
    \right.
\end{equation}
We can solve the latter two equations in \eqref{steadystate} to get
\begin{equation}
\label{h1}
\begin{aligned}
I_2 & = \frac{(N(\beta_2 - \gamma_2) - \beta_2 I_1)(\beta_1 ( 1 - \epsilon) I_1 + N \sigma_2)}{\beta_2 (\beta_1 ( 1 - \epsilon) I_1 + N(\gamma_2 + \sigma_2))} =: \phi(I_1),\\
R_2 & =  \frac{N \gamma_2 (N(\beta_2 - \gamma_2) - \beta_2 I_1)}{\beta_2 (\beta_1 ( 1 - \epsilon) I_1 + N(\gamma_2 + \sigma_2))} =: \psi(I_1)
\end{aligned}
\end{equation}
We note that this steady state only satisfies $I_1 > 0$ and $R_1 > 0$ if $\displaystyle{I_1 <  N \left(1 - \frac{1}{\mathscr{R}_2}\right)}.$

We now substitute the equations $\phi(I_1)$ and $\psi(I_1)$ from \eqref{h1} into the first two equations of \eqref{steadystate}. This gives
\begin{equation}
    \label{steadystate-partial}
    \left\{ \; \; \;
    \begin{aligned}
\frac{\beta_1}{N} (N - I_1 - \phi(I_1) - R_1 - \epsilon \psi(I_1)) I_1 - \gamma_1 I_1 & = 0\\
\gamma_1 I_1 - \sigma_1 R_1 - \frac{\beta_2}{N} R_1 \phi(I_1) & = 0.
    \end{aligned}
    \right.
\end{equation}
In order to show whether these two equations can be satisfied for $I_1 > 0$, we solve for $R_1$ in \eqref{steadystate-partial} in terms of $I_1$ and compute the corresponding derivatives. We have the following:
\begin{equation}
    \label{steadystate-null}
    \left\{ \; \; \;
    \begin{aligned}
        R_1 & = N - I_1 - \phi(I_1) -\epsilon \psi(I_1) - \frac{N \gamma_1}{\beta_1} =: f(I_1) \\
        R_1 & = \frac{N \gamma_1 I_1}{N \sigma_1 + \beta_2 \phi(I_1)} =: g(I_1)
    \end{aligned}
    \right.
\end{equation}
Computing and simplifying the derivatives gives the following:
\tiny
\[
    \left\{ \; \; \;
    \begin{aligned}
        f'(I_1) & =  -\frac{ N^2\gamma_2 (1 - \epsilon)((1 - \epsilon)(\beta_2 - \gamma_2) \beta_1 + \beta_2 (\gamma_2 + \sigma_2))}{\beta_2 (\beta_1 I_1 (1 - \epsilon) + N (\gamma_2 + \sigma_2))^2}\\
        g'(I_1) & = \frac{N \gamma_1 ([\beta_2 (I_1^2(1 - \epsilon)^2 \beta_1^2 + 2N I_1 \sigma_2 (1 - \epsilon) \beta_1 + N^2 \sigma_2 (\gamma_2 + \sigma_2))] \phi(I_1) +( \beta_1I_1(1 - \epsilon) + N \sigma_2)(\beta_2 \beta_1(1 - \epsilon)I_1^2 + (\sigma_1 (1 - \epsilon) \beta_1 + \beta_2 \sigma_2)N I_1 + N^2 \sigma_1(\gamma_2 + \sigma_2)) )}{(N \sigma_1 + \beta_2 \phi(I_1))^2}.
    \end{aligned}
    \right.
\]
\normalsize
Since $1-\epsilon \geq 0$, $\beta_2 > \gamma_2$, and $\phi(I_1) \geq 0$, we have that $f'(I_1) < 0$ and $g'(I_1) > 0$.

We observe that $g(0) = 0$. 
In order to have a solution in the region $I_1 > 0$ and $R_1 > 0$, it is necessary that $f(0) > 0$, which corresponds to
\[\displaystyle{\mathscr{R}_1  > \mathscr{R}_2 \left( \frac{\gamma_2 + \sigma_2 }{(1-\epsilon)\beta_2 + \epsilon \gamma_2  + \sigma_2} \right)} \Longrightarrow \displaystyle{\frac{\mathscr{R}_1}{\mathscr{R}_2} \left( \frac{(1-\epsilon)\beta_2 + \epsilon \gamma_2  + \sigma_2}{\gamma_2 + \sigma_2 } \right)} > 1. \]
We now check where that solution can be. We will have a solution in the region $\displaystyle{0 < I_1 < N \left( 1 - \frac{1}{\mathscr{R}_2} \right) =: I_1^*}$, corresponding to $I_2 > 0$, if and only if $g(I_1^*)> f(I_1^*)$. Substituting $I_1 = I_1^*$ into \eqref{steadystate-null} gives the condition
\[\displaystyle{\mathscr{R}_2 > \mathscr{R}_1\left( \frac{\gamma_1 + \sigma_1}{\beta_1 + \sigma_1} \right) } \Longrightarrow \displaystyle{\frac{\mathscr{R}_2}{\mathscr{R}_1} \left( \frac{\beta_1 + \sigma_1}{\gamma_1 + \sigma_1} \right) } > 1 .\]
It follows that a solution to \eqref{steadystate-partial} with $I_2 > 0$ and $I_1 > 0$ if and only if $\min \{ \mathscr{R}_1, \mathscr{R}_2, \mathscr{R}_{12}, \mathscr{R}_{21} \} > 1$ is satisfied.


\end{proof}

\section{Proof of Theorem \ref{theorem3}}
\label{app:b}

We start our analysis of \eqref{SIR-piecewise} by deriving a few properties of the original strain steady state function $\omega(I_2,R_2)$ \eqref{omega}. We have the following Lemma.

\begin{lem}
\label{lemma3}
The original strain steady state function $\omega(I_2,R_2)$ \eqref{omega} satisfies the following properties:
\begin{enumerate}
\item
$\displaystyle{\frac{\partial \omega}{\partial I_2} = \left\{
    \begin{array}{rlrl}
        & \displaystyle{N \gamma_1\left( \frac{1 + \displaystyle{\frac{ \beta_2 \omega(I_2,R_2)}{\beta_2 I_2 + N \sigma_1}}}{ \beta_2 I_2 + N (\gamma_1 + \sigma_1)}\right)-1,} & & \mbox{   if } I_2 + \epsilon R_2 < N \displaystyle{\left(1 - \frac{1}{\mathscr{R}_1}\right)} \\
        & 0, & & \mbox{   if } I_2 + \epsilon R_2 > N \displaystyle{\left(1 - \frac{1}{\mathscr{R}_1}\right)}
    \end{array}
    \right.}$
\item
$\displaystyle{\frac{\partial \omega}{\partial R_2} = \left\{
    \begin{array}{rlrl}
        & \displaystyle{-\epsilon \left( \frac{\beta_2 I_2 + N \sigma_1}{\beta_2 I_2 + N (\gamma_1 + \sigma_1)} \right)}, 
        & & \mbox{   if } I_2 + \epsilon R_2 < N \displaystyle{\left(1 - \frac{1}{\mathscr{R}_1}\right)} \\
        & 0, & & \mbox{   if } I_2 + \epsilon R_2 > N \displaystyle{\left(1 - \frac{1}{\mathscr{R}_1}\right)}
    \end{array}
    \right.}$
\item
\emph{(a)} $\displaystyle{\frac{\partial \omega}{\partial I_2} + 1 > 0}$; and \emph{(b)} $\displaystyle{\frac{\partial \omega}{\partial R_2} + 1 > 0}$.
\end{enumerate}
\end{lem}

\begin{proof}
Properties 1. and 2. can be verified by directly differentiating \eqref{omega} and noting that $\omega(I_2,R_2)$ is not differentiable at $I_2 + \epsilon R_2 = N \displaystyle{\left( 1 - \frac{1}{\mathscr{R}_1}\right)}$ so that all inequalities are strict.

We now consider the three parts which compose property 3. We notice that they are trivially true for $\displaystyle{I_2 + \epsilon R_2 > N \displaystyle{\left(1 - \frac{1}{\mathscr{R}_1}\right)}}$, so we only need to consider $I_2 + \epsilon R_2 < N \displaystyle{\left(1 - \frac{1}{\mathscr{R}_1}\right)}$. By properties 1. and 2., we have
\begin{enumerate}
\item[(a)]
$\displaystyle{\frac{\partial \omega}{\partial I_2} + 1 = \displaystyle{N \gamma_1\left( \frac{1 + \displaystyle{\frac{ \beta_2 \omega(I_2,R_2)}{\beta_2 I_2 + N \sigma_1}}}{ \beta_2 I_2 + N (\gamma_1 + \sigma_1)}\right)} > 0},$
\item[(b)]
$\displaystyle{\frac{\partial \omega}{\partial R_2} + 1} = \displaystyle{-\epsilon \left( \frac{\beta_2 I_2 + N \sigma_1}{\beta_2 I_2 + N (\gamma_1 + \sigma_1)} \right)+1} = \frac{(1-\epsilon)(\beta_2 I_2 + N \sigma_1) + N \gamma_1}{\beta_2 I_2 + N (\gamma_1 + \sigma_1)} > 0$,
\end{enumerate}
because $\omega(I_2,R_2) \geq 0$ by \eqref{omega} and $1 - \epsilon \geq 0$ since $0 \leq \epsilon \leq 1$, and we are done.
\end{proof}

We now prove Theorem \ref{theorem3}.

\begin{proof} 

\noindent \emph{Proof of 1:} We consider the properties of the $I_2$- and $R_2$-nullcline in the region int$(\Lambda)$. We have the following:
\begin{eqnarray}
\label{inullcline}
I_2' = 0 & \Longrightarrow & f(I_2,R_2) = \beta_2 - \gamma_2 - \frac{\beta_2}{N} (\omega(I_2,R_2) + I_2 + R_2) = 0\\
\label{rnullcline}
R_2'= 0 & \Longrightarrow & g(I_2,R_2) = \gamma_2 I_2 - \sigma_2 R_2 - \frac{\beta_1}{N} (1-\epsilon) \omega(I_2,R_2) R_2 = 0.
\end{eqnarray}
where $\omega(I_2,R_2)$ is given by \eqref{omega}. Also note that int$(\Lambda)$ we have $I_2 > 0$ so that we have divided by $I_2$ in the $I_2$-nullcline. Due to the nonlinearities in $I_2$ and $R_2$ in \eqref{inullcline} and \eqref{rnullcline}, we will analyze $f(I_2,R_2)$ and $g(I_2,R_2)$ from implicit form.

In order to determine the $R_2$-intercepts, we note that $\mathscr{R}_{21} > 1$ implies $\beta_2 \gamma_1 (\beta_1 + \sigma_1) - \beta_1 \gamma_2 ( \gamma_1 +\sigma_1)> 0$. We now evaluate $f(I_2,R_2)$ and $g(I_2,R_2)$ at $I_2 = 0$. This gives
\[
\begin{aligned}
f(0,R_2) = & \; \beta_2 - \gamma_2 - \frac{\beta_2}{N} (\omega(0,R_2) + R_2) = 0 \\
\Longrightarrow & \; \beta_2 - \gamma_2 - \frac{\beta_2}{N} \left(\frac{\sigma_1\left(N(\beta_1 - \gamma_1) - \epsilon \beta_1 R_2\right)}{\beta_1(\gamma_1 + \sigma_1)} + R_2\right) \\
\Longrightarrow & \; R_2 = \frac{N(\beta_2 \gamma_1 (\beta_1 + \sigma_1) - \beta_1 \gamma_2 ( \gamma_1 +\sigma_1))}{\beta_1 \beta_2((1-\epsilon) \sigma_1 + \gamma_1)} > 0 \\
g(0,R_2) = & \; - \left( \sigma_2 + \frac{\beta_1}{N} (1-\epsilon) \omega(0,R_2)\right) R_2 = 0 \\
\Longrightarrow & \; R_2 = 0
\end{aligned}
\]
where we have utilized $1 - \epsilon \geq 0$ and $\omega(I_2,R_2) \geq 0$. 
It follows that, if $\mathscr{R}_{21} > 1$ is satisfied, then the $I_2$-nullcline has a positive $R_2$-intercept and the $R_2$-nullcline has the $R_2$-intercept $R_2=0$.

We now want to determine how the nullclines change as $I_2$ increases in the region $I_2 > 0$. We consider $R_2 = R_2(I_2)$ as a function of $I_2$ and implicitly differentiate to find $\displaystyle{\frac{dR_2}{dI_2}}$ on the surfaces $f(I_2,R_2)$ and $g(I_2,R_2)$. For the $I_2$-nullcline $f(I_2,R_2)$ \eqref{inullcline}, we have
\[\frac{d}{dI_2} f(I_2,R_2(I_2)) = -\frac{\beta_2}{N} \left( \frac{\partial \omega}{\partial I_2} + \frac{\partial \omega}{\partial R_2} \frac{dR_2}{dI_2} + 1 + \frac{dR_2}{dI_2} \right) = 0.\]
Solving for $\displaystyle{\frac{dR_2}{dI_2}}$ gives
\[
\displaystyle{\frac{dR_2}{dI_2} = - \left( \frac{\displaystyle{\frac{\partial \omega}{\partial I_2} + 1}}{\displaystyle{\frac{\partial \omega}{\partial R_2} + 1}} \right)} < 0
\]
by properties 3(a) and 3(b) of Lemma \ref{lemma3}. 
It follows that the $I_2$-nullcline is strictly decreasing in int$(\Lambda)$. 

For the $R_2$-nullcline $g(I_2,R_2)$ \eqref{rnullcline} we have
\[
\frac{d}{dI} g(I_2,R_2(I_2)) = \gamma_2-\sigma_2 \frac{dR_2}{dI_2} - \frac{\beta_1}{N} (1-\epsilon) \left( \frac{\partial \omega}{\partial I_2} R + \frac{\partial \omega}{\partial R_2} \frac{dR_2}{dI_2} R + \omega(I_2,R_2) \frac{dR_2}{dI_2} \right) = 0.
\]
Solving for $\displaystyle{\frac{dR_2}{dI_2}}$ gives
\begin{equation}
\label{rnull-diff}
\frac{dR_2}{dI_2} = \frac{\displaystyle{-\beta_1 (1- \epsilon) \frac{\partial \omega}{\partial I_2} R_2 + \gamma_2 N}}{\displaystyle{\beta_1 ( 1 - \epsilon) \left( \frac{\partial \omega}{\partial R_2} R_2 + \omega(I_2,R_2) \right) + \sigma_2 N}}
\end{equation}
We now consider cases. If $\epsilon = 0$, then \eqref{rnull-diff} reduces to
\begin{equation}
\label{eqn100}
\frac{dR_2}{dI_2} = \frac{\displaystyle{-\beta_1 \frac{\partial \omega}{\partial I_2} R_2 + \gamma_2 N}}{\displaystyle{\beta_1 \omega(I_2,R_2) + \sigma_2 N}}
\end{equation}
where we have note that $\displaystyle{\frac{\partial \omega}{\partial R_2} = 0}$ if $\epsilon = 0$ from property 2. of Lemma \ref{lemma3}. Since $\omega(I_2,R_2) \geq 0$ we have $\beta_1 \omega(I_2,R_2) + \sigma_2 N > 0$, so we only need to consider the numerator of \eqref{eqn100}. We first note that \eqref{rnullcline} implies that
\begin{equation}
\label{eq14}
\displaystyle{\gamma_2 = \frac{R_2}{I_2} \left( \sigma_2 + \frac{\beta_1}{N}(1-\epsilon)R_2\omega(I_2,R_2)\right).}
\end{equation}
Substituting \eqref{eq14} and the form of $\displaystyle{\frac{\partial \omega}{\partial I_2}}$ from property 1. of Lemma \ref{lemma3} into \eqref{eqn100} gives and factoring by $\omega(I_2,R_2)$ gives
\[\frac{dR_2}{dI_2} = \frac{1}{I_2} \left( \alpha_1 \omega(I_2,R_2) + \alpha_2 \right)\]
where
\[
\begin{aligned}
\alpha_1 & = \displaystyle{\frac{\beta_1 R (\sigma_1(\gamma_1 + \sigma_1)N^2 + 2 I_2 N \beta_2 \sigma_1 + I^2 \beta_2^2)}{ (\beta_2 I_2 + N \sigma_2)(\beta_2 I_2 + N(\gamma_1 + \gamma_1))}} > 0\\
\alpha_2 & = \frac{\sigma_2 (\gamma_1 + \sigma_1)N^2 + (\beta_1 \sigma_1 + \beta_2 \sigma_2) I_2 N + I^2 \beta_1 \beta_2}{\beta_2 I_2 + N(\gamma_1 + \gamma_1)} > 0
\end{aligned}
\]
It follows that, for the $R_2$-nullcline we  have $\displaystyle{\frac{dR_2}{dI_2}} > 0$ if $\epsilon = 0$. For the case $\epsilon = 1$, we have that \eqref{rnull-diff} reduces to
\[\frac{dR_2}{dI_2} = \frac{\gamma_2}{\sigma_2} > 0.\]
It follows that the $R_2$-nullcline is increasing in the region int$(\Lambda)$ if $\epsilon = 0$ or $\epsilon = 1$.

We now consider the directions in the vector field for different regions of the planar state space $(I_2,R_2)$. We show that $I_2' = f(I_2,R_2)$ and $R_2' = g(I_2,R_2)$ strictly decrease as $R_2$ increases. From \eqref{SIR-piecewise}, we have that
\[\frac{\partial}{\partial R_2} f(I_2,R_2) = \frac{\partial}{\partial R_2} \left[ \frac{\beta_2}{N} (N - \omega(I_2,R_2) - I_2 - R_2) I_2 - \gamma_2 I_2 \right] = -\frac{\beta_2}{N}I_2\left(\frac{d\omega}{dR_2} + 1\right)<0\]
by property 3(b) of Lemma \ref{lemma3}. We also have
\[
\begin{aligned}
\frac{\partial}{\partial R_2} g(I_2,R_2) & = \frac{\partial}{\partial R_2} \left[ \gamma_2 I_2 - \sigma_2 R_2 - \frac{\beta_1}{N} (1 - \epsilon) \omega(I_2,R_2) R_2 \right] \\
& = -\sigma_2 - \frac{\beta_1}{N} (1 - \epsilon) \left( \frac{d \omega}{dR} R_2 + \omega(I_2,R_2) \right).\\
& = \left\{ \begin{array}{ll} \displaystyle{-\sigma_2 - \frac{\beta_2}{N} \omega(I_2,R_2)} < 0, & \mbox{if } \epsilon = 0 \\
\displaystyle{-\sigma_2} < 0, & \mbox{if } \epsilon = 1 \end{array} \right.
\end{aligned}
\]
where we have noted that $\displaystyle{\frac{\partial \omega}{\partial R_2} = 0}$ for $\epsilon = 0$ by property of 2. of Lemma \ref{lemma3}. It follows that, if $\epsilon = 0$ or $\epsilon =1$, we have that $I_2' < 0$ and $R_2' < 0$ if we are above both the $I_2$- and $R_2$-nullcline. The rest of the cases follow similarly. \\

\noindent \emph{Proof of 2:} Since the $I_2$-nullcline $f(I_2,R_2) = 0$ is strictly decreasing from a positive $R_2$-intercept, the $R_2$-nullcline $g(I_2,R_2) = 0$ is strictly increasing from an $R_2$-intercept of $R_2=0$, and solutions are restricted to $\Lambda$ by Theorem \ref{theorem1}, we have that the nullclines must have a unique intersection in int$(\Lambda)$. Local exponential stability follows from Lemma \ref{lem:linearstability} in Appendix \ref{app:d}.

For global stability, we utilize the Poincar\'{e}-Bendixson Theorem \cite{wiggins2003}. Since the system is planar, it may not contain chaotic trajectories \cite{wiggins2003}. The only remaining possibilities are convergence to the unique steady state or convergence to a limit cycle. We will rule out the existence a limit cycle using the modified Dulac's Criterion (Lemma \ref{dulac} in Appendix \ref{app:b}).

Consider the piecewise-defined vector field \eqref{SIR-piecewise} and the function $\displaystyle{\psi(I_2,R_2) = \frac{1}{I_2}}$, which is continuous on int$(\Lambda)$. We notice that \eqref{SIR-piecewise} is continuous on int$(\Lambda)$ and continuously differentiable everywhere on $\Lambda$ except $\displaystyle{I_2 + \epsilon R_2 = N \left( 1 - \frac{1}{\mathscr{R}_1} \right)}$.

We have that
\[\begin{aligned}
& \frac{\partial}{\partial I_2} [ \psi(I_2,R_2) f(I_2,R_2) ] + \frac{\partial}{\partial R_2} [ \psi(I_2,R_2) g(I_2,R_2) ] \\
& = \frac{\partial}{\partial I_2} \left[ \frac{\beta_2}{N} (N - \omega(I_2,R_2) - I_2 - R_2) - \gamma_2 \right] + \frac{\partial}{\partial R_2} \left[ \gamma_2 - \left( \frac{N\sigma_2 + \beta_1 (1 -\epsilon) \omega(I_2,R_2)}{NI_2} \right) R_2 \right]\\
& = -\frac{\beta_2}{N} \left( \frac{\partial \omega}{\partial I_2} + 1 \right) - \frac{\sigma_2}{I_2} - \frac{\beta_1 (1 - \epsilon)}{NI_2} \left( \frac{\partial \omega}{\partial R_2} R_2 + \omega(I_2,R_2) \right) \\
& = \left\{ \begin{array}{ll} \displaystyle{-\frac{\beta_2}{N} \left( \frac{\partial \omega}{\partial I_2} + 1 \right) - \frac{\sigma_2}{I_2} - \frac{\beta_1}{NI_2} \omega(I_2,R_2)} < 0, & \mbox{if } \epsilon = 0 \\
\displaystyle{-\frac{\beta_2}{N} \left( \frac{\partial \omega}{\partial I_2} + 1 \right) - \frac{\sigma_2}{I_2}} < 0, & \mbox{if } \epsilon = 1 \end{array} \right.
\end{aligned} \]
It follows by Lemma \ref{dulac} in Appendix \eqref{app:d} that we cannot have a periodic orbit lying entirely within $\Lambda$. Consequently, from Poincar\'{e}-Bendixson Theorem $(\bar{I}_2,\bar{R}_2)$ is the global attractor for trajectories starting in int$(\Lambda)$.

We now establish the location of the unique steady state $(\bar{I}_2,\bar{R}_2) \in \mbox{int}(\Lambda)$. We note that, since the $I_2$-nullcline is above the $R_2$-nullcline at $I_2=0$, for them to intersect in $\displaystyle{0 < I_2 + \epsilon R_2 < N \left( 1 - \frac{1}{\mathscr{R}_2} \right)}$ it is sufficient for the $R_2$-nullcline to be above the $I_2$-nullcline along $\beta_2 I_2 + \beta_2 \epsilon R = N \left( \beta_2 - \gamma_2 \right)$. Note that we have $\omega(I_2,R_2) = 0$ along this set. It consequently remains to consider the equations $f(I_2,R_2) = 0$ and $g(I_2,R_2) = 0$ along this set.

For the intersection of $f(I_2,R_2) = 0$ \eqref{inullcline} and $\beta_2 I_2 + \beta_2 \epsilon R = N \left( \beta_2 - \gamma_2 \right)$, we solve the following system of equations
\[\left\{
\begin{array}{rlrll}
\beta_1 I_2 & + & \beta_1 \epsilon R_2 & = & N(\beta_1 - \gamma_1)\\
\beta_2 I_2 & + & \beta_2 R_2 & = & N(\beta_2 - \gamma_2)
\end{array}
\right.\]
to get
\[R_f = \frac{N( \beta_2 \gamma_1-\beta_1 \gamma_2)}{\beta_2 \beta_1(1-\epsilon)}.\]
For the intersection of $g(I_2,R_2) = 0$ \eqref{rnullcline} and $\beta_2 I_2 + \beta_2 \epsilon R = N \left( \beta_2 - \gamma_2 \right)$, we solve the following system of equations
\[\left\{
\begin{array}{rlrll}
\beta_1 I_2 & + & \beta_1 \epsilon R_2 & = & N(\beta_1 - \gamma_1)\\
\gamma_2 I_2 & - & \sigma_2 R_2 & = & 0
\end{array}
\right.\]
to get
\[R_g = \frac{N \gamma_2(\beta_1 - \gamma_1)}{\beta_1(\sigma_2 + \epsilon \gamma_2)}.\]
In order to have $R_g > R_f$ we must have
\small
\[
\beta_2\gamma_2(1-\epsilon)(\beta_1 - \gamma_1)  > (\sigma_2 + \epsilon \gamma_2)( \beta_2 \gamma_1-\beta_1 \gamma_2) \Longrightarrow \frac{\beta_1}{\gamma_1} > \frac{\beta_2}{\gamma_2} \left( \frac{\gamma_2 + \sigma_2}{(1-\epsilon) \beta_2 + \epsilon \gamma_2 + \sigma_2} \right)\Longrightarrow \frac{\mathscr{R}_1}{\mathscr{R}_2} \left( \frac{(1-\epsilon) \beta_2 + \epsilon \gamma_2 + \sigma_2}{\gamma_2 + \sigma_2} \right) > 1.
\]
\normalsize
It follows that if $\mathscr{R}_{12} > 1$, then the steady state $(\bar{I}_2,\bar{R}_2)$ satisfies $\displaystyle{0 < \bar{I}_2 + \epsilon \bar{R}_2 < N \left( 1 - \frac{1}{\mathscr{R}_1} \right)}$ which is sufficient to guarantee $\omega(\bar{I}_2,\bar{R}_2)> 0$ by \eqref{omega}. If $\mathscr{R}_{12} < 1$, however, the intersection will satisfy $\displaystyle{\bar{I}_2 + \epsilon \bar{R}_2 > N \left( 1 - \frac{1}{\mathscr{R}_1} \right)}$ so that $\omega(\bar{I}_2,\bar{R}_2) = 0$ by \eqref{omega}, and we are done.

\end{proof}

\section{Supporting results}
\label{app:d}

\begin{lem}
\label{lem:linearstability}
Consider the nonlinear planar system
\begin{equation}
\label{nonlinear}
\left\{
\begin{aligned}
\frac{dx}{dt} & = f(x,y) \\
\frac{dy}{dt} & = g(x,y).
\end{aligned}
\right.
\end{equation}
with a hyperbolic fixed point $(\bar{x},\bar{y})$. Suppose that, at the fixed point $(\bar{x},\bar{y})$, the $x$-nullcline is decreasing, the $y$-nullcline is increasing, and the region above $(\bar{x},\bar{y})$ satisfies $x'<0$ and $y'<0$. Then $(\bar{x},\bar{y})$ is locally exponentially stable.
\end{lem}

\begin{proof}
Since $(\bar{x},\bar{y})$ is a hyperbolic fixed point by assumption, the Hartman-Grobman Theorem \cite{grobman1959,hartman1960} guarantees the existence of a neighborhood around $(\bar{x},\bar{y})$ where trajectories are mapped to a corresponding linear system with coefficient matrix given by the Jacobian of the nonlinear system evaluated at the fixed point. We furthermore have tangency of the nonlinear and nonlinear system nullclines at the fixed point and that $x'<0$ and $y'<0$ above the nullclines in the linear system. We now prove stability in the linearized system. 

Consider the general linear system of ordinary differential equations in two variables:
\begin{equation}
\label{linear}
\left\{
\begin{aligned}
\frac{dx}{dt} & = ax + by \\
\frac{dy}{dt} & = cx + dy.
\end{aligned}
\right.
\end{equation}
In order to have $x'<0$ above the $x$-nullcline, we require that:
\[x'<0 \Longrightarrow ax+by < 0 \Longrightarrow by < -ax \Longrightarrow y > -\frac{a}{b}x.\]
This can only happen if $b < 0$. A similar argument for the $y$-nullcline gives that $d < 0$.

We now consider conditions 1 and 2. The $x$-nullcline is given by:
\[x'=0 \Longrightarrow ax+by = 0 \Longrightarrow y = - \frac{a}{b} x.\]
In order for this to have a negative slope, we require that $-\frac{a}{b}<0$. Given that $b < 0$, we must also have $a < 0$. Similarly, for the $y$-nullcline, we have:
\[y'=0 \Longrightarrow cx+dy = 0 \Longrightarrow y = -\frac{c}{d} x.\]
For a positive slope, we require $-\frac{c}{d} > 0$ which, given that $d < 0$, implies that $d > 0$.

It follows that the coefficient matrix for the system has the indicated sign pattern
\[A = \left[ \begin{array}{cc} a & b \\ c & d \end{array} \right] = \left[ \begin{array}{cc} - & - \\ + & - \end{array} \right].\]
It immediately follows that:
\[\mbox{Tr}(A) = a + d = (-) + (-) < 0\]
and 
\[\mbox{Det}(A) = ad - bc = (-)(-) - (-)(+) > 0.\]
It follows from classical dynamical systems theory that the steady state $(0,0)$ of the linear system \eqref{linear} is exponentially stable. Local exponential stability of $(\bar{x},\bar{y})$ for \eqref{nonlinear} follows, and we are done.
\end{proof}

\begin{lem}[Modified Dulac's Criterion]
\label{dulac}
Consider the nonlinear planar system \eqref{nonlinear} 
and let $\Lambda$ be a simply connected region in this plane. Consider a connected curve $\Gamma$ lying within $\Lambda$ which divides $\Lambda$ into two regions, $\Lambda_1$ and $\Lambda_2$, (i.e. $\Lambda_1 \cup \Lambda_2 = \Lambda$ and $\Lambda_1 \cap \Lambda_2 = \Gamma$) both of which are simply connected. Suppose that $f$ and $g$ are continuous on $\Lambda$ and continuously differentiable in the interiors of $\Lambda_1$ and $\Lambda_2$ (but potentially not continuously differentiable along $\Gamma$).

Suppose there is a continuously differentiable function $\psi(x,y)$ such that 
\[\frac{\partial}{\partial x} [ \psi(x,y) f(x,y) ] + \frac{\partial}{\partial y} [ \psi(x,y) g(x,y) ]\]
has a constant sign in the interior of $\Lambda_1$ and $\Lambda_2$. Then there are no nonconstant periodic orbits lying entirely in $\Lambda$.
\end{lem}

\begin{proof}
Without loss of generality, we assume that 
\[\frac{\partial}{\partial x} [ \psi(x,y) f(x,y) ] + \frac{\partial}{\partial y} [ \psi(x,y) g(x,y) ] > 0\]
in the interior of $\Lambda_1$ and $\Lambda_2$. Suppose there is a periodic orbit $C$ enclosing a bounded region $D$. Without loss of generality, we assume that $C$ is counter-clockwise oriented around $D$.

The Dulac criterion guarantees that $C$ does not lie entirely in $\Lambda_1$ or $\Lambda_2$ so we need only consider the case where $C$ crosses between $\Lambda_1$ and $\Lambda_2$. We note that $C$ may cross $\Gamma$ an arbitrary number of even times; however, each time it crosses and crosses back, we may generically divide the existing region into two subregions, both of which are simply connected. Consequently, it is sufficient to consider the first pair of crossings and induct to any subsequent pairs of crossings.


Suppose $C$ crosses $\Gamma$ and then crosses back. We define $D_1$ and $D_2$ to be the two open subregions of $D$ created by this pair of crossings such that $\overline{D_1} \cup \overline{D_2} = D$ and $\overline{D_1} \cap \overline{D_2} \subseteq \gamma$. We define $C'$ to be a curve traversing the boundary between $D_1$ and $D_2$, i.e. $C' = \overline{D_1} \cap \overline{D_2}$. This curve can be oriented in either direction. We define $-C'$ to the curve oriented in the other direction. We now define $C_1$ and $C_2$ to be the two counter-clockwise oriented curves which traverse the boundary of $D_1$ and $D_2$ respectively. Note that these curves contain part of the curve $C$ and one of either $C'$ or $-C'$.


Now, since $f(x,y)$ and $g(x,y)$ are continuously differentiable within $D_1$ and $D_2$ we have that
\[\iint_{D_i} \frac{\partial}{\partial x} [ \psi(x,y) f(x,y) ] + \frac{\partial}{\partial y} [ \psi(x,y) g(x,y) ] \; dy dx > 0\]
for $i = 1, 2$. By Green's Theorem, it follows that we have
\[
\begin{aligned}
0 & < \sum_{i=1}^2 \iint_{D_i} \frac{\partial}{\partial x} [ \psi(x,y) f(x,y) ] + \frac{\partial}{\partial y} [ \psi(x,y) g(x,y) ] \; dy dx  \\
& = \sum_{i=1}^2 \oint_{C_i} -\psi(x,y) g(x,y) \; dx + \psi(x,y) f(x,y) \; dy \\
& = \oint_C - \psi(x,y) g(x,y) \; dx + \psi(x,y) f(x,y) \; dy + \int_{C'} -\psi(x,y) g(x,y) \; dx + \psi(x,y) f(x,y) \; dy\\ & \hspace{1in} + \int_{-C'} -\psi(x,y) g(x,y) \; dx + \psi(x,y) f(x,y) \; dy \\
& = \oint_C - \psi(x,y) g(x,y) \; dx + \psi(x,y) f(x,y) \; dy \\
& = \oint_C - \psi(x,y) \frac{dy}{dt} \; dx + \psi(x,y) \frac{dx}{dt} \; dy = 0
\end{aligned}
\]
This is a contradiction, so that we cannot have a periodic orbit which crosses $\Gamma$ and then crosses back. We may now induct to multiple crossings by noting that each subsequent crossing will split one of the regions already considered in exactly the same way as the first split. In particular, it will be produce two simple connected subregions and the boundary between the split subregions may be parametrized by a connected curve. Since the periodic orbit $C$ was chosen arbitrarily in $\Lambda$, we are done.
\end{proof}

\end{document}